\documentclass[11pt]{amsart}
\usepackage{etex}
\usepackage{graphicx}

\usepackage{amssymb}
\usepackage{mathrsfs}
\usepackage{amsmath}
\usepackage{amsthm}
\usepackage{xcolor}
\usepackage{bm}
\usepackage{placeins}
\usepackage{verbatim}
\usepackage{multirow}
\usepackage{float}
\usepackage{url}
\usepackage{enumitem}
\usepackage[a4paper,left=1.3in,right=1.3in,top=1.3in,bottom=1.3in]{geometry}
\usepackage{tikz-cd}
\usepackage[all]{xy}
\usepackage{url}
\usepackage{array,booktabs}
\usepackage{blkarray}
\usepackage{algorithmicx}
\usepackage{algpseudocode}
\usepackage{algorithm}
\usepackage[hidelinks]{hyperref}
\usepackage{hobsub}
\usepackage[width=0.9\textwidth]{caption}
\makeatletter
\DeclareRobustCommand*{\bfseries}{%
	\not@math@alphabet\bfseries\mathbf
	\fontseries\bfdefault\selectfont
	\boldmath
}
\makeatother
\numberwithin{equation}{section}

\newtheorem{proposition}[equation]{Proposition}
\newtheorem{lemma}[equation]{Lemma}
\newtheorem{corollary}[equation]{Corollary}

\newtheorem*{intro_prop:taut_eqal}{Proposition \ref{prop:lower=upper}}
\newtheorem*{intro_prop:graphs}{Proposition \ref{prop:nonisom_graphs}}
\newtheorem*{intro_prop:veering}{Proposition \ref{prop:not_equal}}
\newtheorem*{intro_prop:taut_algorithm}{Proposition \ref{prop:veering_output}}
\newtheorem*{intro_prop:veering_algorithm}{Proposition \ref{prop:big_algorithm_correct}}
\newtheorem*{intro_prop:teich_algorithm}{Proposition \ref{prop:alg_teich_correct}}

\theoremstyle{definition}
\newtheorem{definition}[equation]{Definition}

\newtheorem{remark}[equation]{Remark}
\newtheorem*{remark*}{Remark}
\newtheorem*{Notation*}{Notation}

\newtheorem*{warning*}{Warning}
\newtheorem{example}[equation]{Example}

\newcommand{\ZH}{\zz \lbrack H \rbrack}

\newcommand{\Fit}{\mathrm{Fit}}

\newcommand{\bez}{\setminus}

\newcommand{\zz}{\mathbb{Z}}
\newcommand{\rr}{\mathbb{R}}
\newcommand{\bigslant}[2]{{\raisebox{.2em}{$#1$}\left/\raisebox{-.2em}{$#2$}\right.}}
\newcommand{\Mab}{M^{ab}}
\newcommand{\Xab}{X^{ab}}
\newcommand{\Vab}{\mathcal{V}^{ab}}
\newcommand{\Tab}{\mathcal{T}^{ab}}
\newcommand{\Dab}{\mathcal{D}^{ab}}
\newcommand{\lowert}{\tau^{ab,L}}
\newcommand{\upert}{\tau^{ab,U}}
\newcommand{\face}[1]{\raisebox{.03em}{\large{$\mathtt{#1}$}}}
\newcommand{\module}{\mathcal{E}}

\algrenewcommand{\algorithmiccomment}[1]{\hspace*{\fill}
	\color{gray}\small $\#$  #1 \color{black}\normalsize}

\begin{document}

	\numberwithin{equation}{section}
	
	\title[The taut, the veering and the Teichm\"uller polynomials]{Computation of the taut, the veering and the Teichm\"uller polynomials}
	
	\author{Anna Parlak}
	\address{Mathematics Institute, University of Warwick, Coventry CV4 7AL, United	Kingdom}
	\email{a.parlak@warwick.ac.uk}

	\thanks{This work was supported by The Engineering and Physical Sciences Research Council (EPSRC) under grant EP/N509796/1 studentship 1936817}


	\begin{abstract}
		Landry, Minsky and Taylor [LMT] introduced two polynomial invariants of veering triangulations --- the taut polynomial and the veering polynomial. We give algorithms to compute these invariants. In their definition [LMT] use only the upper track of the veering triangulation, while we consider both the upper and the lower track. We prove that the lower and the upper taut polynomials are equal. However, we show that there are veering triangulations whose lower and upper veering polynomials are different.

		[LMT] related the Teichm\"uller polynomial of a fibred face of the Thurston norm ball with the taut polynomial of the associated layered veering triangulation. We use this result to give an algorithm to compute the Teichm\"uller polynomial of any fibred face of the Thurston norm ball.
	\end{abstract}
	
	\maketitle%
	
	\section{Introduction}
	Let $M$ denote a compact, oriented, connected 3-manifold with a finite, but positive, number of toroidal boundary components. By making boundary components into torus cusps we can represent such a manifold by an  ideal triangulation ---  a decomposition into tetrahedra with vertices removed. Here we consider a special class of ideal triangulations called \emph{(transverse taut) veering}.  They were introduced by Agol as a way to canonically triangulate pseudo-Anosov mapping tori \cite{Agol_veer}. More generally, veering triangulations are tightly connected to pseudo-Anosov flows without perfect fits~\cite{SchleimSegLinks}.
	
	This paper relies heavily on the work of Landry, Minsky and Taylor \cite{LMT}. They defined the \emph{taut polynomial}, the \emph{veering polynomial} and the \emph{flow graph} of a  (transverse taut) veering triangulation. 
	Their definitions refer to canonical train tracks embedded in the 2-skeleton of a veering triangulation called the \emph{upper} and \emph{lower tracks}; see Subsection \ref{subsec:dual_tracks}. The authors of  \cite{LMT} always work with the upper track. In this paper we consider both train tracks. Given a veering triangulation we compare its lower and upper taut polynomials, its lower and upper veering polynomials and its lower and upper flow graphs. 
	
	\begin{intro_prop:taut_eqal}
		Let $\mathcal{V}$ be a veering triangulation. The lower taut polynomial $\Theta^L_{\mathcal{V}}$ of $\mathcal{V}$ is (up to a unit) equal to the upper taut polynomial $\Theta^U_{\mathcal{V}}$ of~$\mathcal{V}$. 
	\end{intro_prop:taut_eqal}
	Hence there is only one taut polynomial associated to a veering triangulation. However, a similar statement does \emph{not} hold for the flow graphs and the veering polynomials. We give examples which show the following.
	\begin{intro_prop:graphs}
		There exists a veering triangulation $\mathcal{V}$ whose lower and upper flow graphs are not isomorphic. 
	\end{intro_prop:graphs}
	\begin{intro_prop:veering}
		There exists a veering triangulation $\mathcal{V}$ for which the lower veering polynomial $V^L_\mathcal{V}$ is not equal to the upper veering polynomial $V^U_\mathcal{V}$ up to a unit.
	\end{intro_prop:veering}
	Thus to a triangulation with a veering structure we can canonically associate a \emph{pair} of flow graphs and a \emph{pair} of veering polynomials.
	
	The majority of the results of this paper are computational. We simplify the original presentation for the  taut module. As a result, the computation of the taut polynomial requires computing only linearly many minors of a matrix, instead of exponentially many  (Proposition~\ref{prop:non-tree_generate}). Using this we give algorithm \texttt{TautPolynomial}. We prove
	\begin{intro_prop:taut_algorithm}
	The output of \emph{\texttt{TautPolynomial}} applied to a veering triangulation $\mathcal{V}$ is equal to the taut polynomial $\Theta_{\mathcal{V}}$ of $\mathcal{V}$.
	\end{intro_prop:taut_algorithm}
	
	We also give algorithm \texttt{LowerVeeringPolynomial} and prove
	\begin{intro_prop:veering_algorithm}
		The output of \emph{\texttt{LowerVeeringPolynomial}} applied to a veering triangulation $\mathcal{V}$ is equal to the lower veering polynomial $V^L_{\mathcal{V}}$ of $\mathcal{V}$.
	\end{intro_prop:veering_algorithm} 
	The upper veering polynomial of $\mathcal{V}$ can be computed as the lower veering polynomial of the veering triangulation obtained from $\mathcal{V}$ by reversing the coorientation on its 2-dimensional faces; see Remark \ref{remark:changes_on_tracks}.
	
	When a veering triangulation $\mathcal{V}$ of $M$ carries a fibration over the circle, we can associate to it a fibred face $\face{F}$ of the Thurston norm ball in $H_2(M, \partial M;\rr)$ \cite[Proposition~2.7]{MinskyTaylor}. 	Face $\face{F}$ is \emph{fully-punctured}. 
	By \cite[Theorem 7.1]{LMT} the  taut polynomial $\Theta_{\mathcal{V}}$ of $\mathcal{V}$ is equal to the Teichm\"uller polynomial~$\Theta_{\mathtt{F}}$ of  $\face{F}$. 	 
	Thus \texttt{TautPolynomial} is an  algorithm to compute the Teichm\"uller polynomial of a fully-punctured fibred face. As observed in \cite[Subsection~7.2]{LMT} this can be generalised to all fibred faces via puncturing. 
	We  present the details of this in algorithm \texttt{Teichm\"ullerPolynomial} and prove
	\begin{intro_prop:teich_algorithm}
		Let $\psi:S \rightarrow S$ be a pseudo-Anosov homeomorphism. Denote by $N$  its mapping torus. Let $\face{F}$ be the fibred face of the Thurston norm ball in $H_2(N, \partial N;\rr)$ such that  $\lbrack S \rbrack \in \rr_+ \hspace{-0.1cm}\cdot \face{F}$. Then the output of \emph{\texttt{Teichm\"ullerPolynomial}}$(\psi)$ is equal to the Teichm\"uller polynomial $\Theta_{\emph{\texttt{F}}}$ of~$\face{F}$.
	\end{intro_prop:teich_algorithm}
	
	Different algorithms to compute the Teichm\"uller polynomial in some special cases have been previously developed. In \cite{LanVal} Lanneau and Valdez present an algorithm to compute the Teichm\"uller polynomial of punctured disk bundles. The authors of \cite{BWKJ}  compute the Teichm\"uller polynomial of \emph{odd-block surface} bundles.  In~\cite{BilletLechti} Billet and Lechti cover the case of alternating-sign Coxeter links. 
	
	The algorithm presented in this paper is general and in principle can be applied to any hyperbolic, orientable 3-manifold fibred over the circle. Moreover, it has already been implemented. The source code is available at~\cite{VeeringCensus}.

	\subsection*{Pseudocodes} 
	\begin{enumerate}[label=\arabic*., wide=0pt, leftmargin=*]
		\item \texttt{FacePairings} (Subsection \ref{subsec:alg_encoding})
		
		\noindent Encodes a triangulation of a free abelian cover of $M$ induced by a transverse taut triangulation of $M$
		
		\item \texttt{TautPolynomial} (Subsection \ref{subsec:computation})
		
		\noindent Computes the taut polynomial $\Theta_{\mathcal{V}}$ of a veering triangulation $\mathcal{V}$
		\item \texttt{LowerVeeringPolynomial} (Subsection \ref{subsec:comp_big})
		
		\noindent Computes the lower veering polynomial $V^L_\mathcal{V}$ of a veering triangulation~$\mathcal{V}$
		
		\item \texttt{BoundaryCycles} (Subsection \ref{subsec:boundary_cycles})
		
		\noindent Finds simplicial 1-cycles in the dual graph of a transverse taut triangulation which are homologous to the boundary components of a surface carried by the 2-skeleton of this triangulation
		
		\item \texttt{Teichm\"ullerPolynomial} (Subsection \ref{subsec:Teich_comp})
		
		\noindent Computes the Teichm\"uller polynomial $\Theta_{\texttt{F}}$ of a fibred face $\face{F}$ of the Thurston norm ball
	\end{enumerate}

	All pseudocodes have been implemented by the author, Saul Schleimer and Henry Segerman. The source codes are available at~\cite{VeeringCensus}.

	\subsection*{Acknowledgements} I am grateful to Samuel Taylor for explaining to me his work on the veering polynomial during my visit at Temple University in July 2019 and subsequent conversations. I thank Saul Schleimer and Henry Segerman for their generous assistance in implementing the algorithms presented in this paper. I also thank Mark Bell for answering my questions about \texttt{flipper}.
	
	This work has been written during PhD studies of the author at the University of Warwick under the supervision of Saul Schleimer, funded by the EPSRC. 
	\section{Transverse taut veering triangulations}\label{sec:pre}
	\subsection*{Ideal triangulations of 3-manifolds}\label{subsec:triangulations}
	An \emph{ideal triangulation} of an oriented 3-manifold~$M$ with torus cusps is a decomposition of $M$ into ideal tetrahedra. Let $n$ be the number of tetrahedra in this decomposition. Since~$M$ has zero Euler characteristic, the number of 2-dimensional faces and the number of edges of the triangulation equals $2n$ and $n$, respectively. The 2-dimensional faces are often called \emph{triangles}. Every triangle of a triangulation has two \emph{embeddings} into two, not necessarily distinct, tetrahedra. The number of (embeddings of) triangles attached to an edge is called the \emph{degree} of this edge. An edge of degree $d$ has $d$ embeddings into triangles and~$d$ embeddings into tetrahedra. 
	
	We denote an ideal triangulation by $\mathcal{T} = (T, F, E)$, where $T, F , E$ denote the set of tetrahedra, triangles and edges, respectively.
	Note that by \emph{edges of a triangle/tetrahedron} or \emph{triangles of a tetrahedron} we mean embeddings of these ideal simplices into the boundary of a higher dimensional ideal simplex.  Similarly, by \emph{triangles/tetrahedra attached to an edge} we mean triangles/tetrahedra in which the edge is embedded, together with this embedding. 
	When we claim that two lower dimensional simplices of a higher dimensional simplex are different we mean that at least their embeddings are different.
	
	\subsection*{Compact and cusped models}
	If we truncate the corners of ideal tetrahedra of $\mathcal{T}$ we obtain a compact 3-manifold with toroidal boundary components. The interior of this manifold is homeomorphic to $M$. For notational simplicity we also denote it by $M$. We freely alternate between the \emph{cusped} and the \emph{compact} models of $M$.  The main advantage of the latter is that we can consider curves in the boundary $\partial M$. 
	\subsection*{Transverse taut triangulations}
	An ideal tetrahedron $t$ is \emph{transverse taut} if each of its faces is cooriented so that  the coorientation on two of the faces of $t$ points into $t$ and on the remaining two faces it points out $t$ \cite[Definition 1.2]{veer_strict-angles}. We call the pair of faces whose coorientations point out of $t$ the \emph{top faces} of~$t$ and the pair of faces whose coorientations point into $t$ the \emph{bottom faces} of $t$. We also say that $t$ is \emph{immediately below} its top faces and \emph{immediately above} its bottom faces. 
	
	We draw a transverse taut tetrahedron as a quadrilateral with two diagonals --- one on top of the other; see Figure~\ref{fig:veering_tetra}. The \emph{top diagonal} is the common edge of the two top faces of $t$ and the \emph{bottom diagonal} is the common edge of the two bottom faces of~$t$. The remaining four edges of a transverse taut tetrahedron are called its \emph{equatorial edges}. This way of presenting a transverse taut tetrahedron $t$ naturally endows it with an abstract assignment of angles from $\lbrace 0, \pi\rbrace$ to its edges. Angle $\pi$ is assigned to both diagonal edges of $t$ and angle 0 is assigned to all its equatorial edges. Such an assignment of angles is called a \emph{taut structure} on $t$ \cite[Definition~1.1]{veer_strict-angles}.
	
	A triangulation $\mathcal{T}=(T, F, E)$ is \emph{transverse taut} if \begin{itemize}
		\item every ideal triangle $f \in F $ is assigned a coorientation so that every ideal tetrahedron $t \in T$ is transverse taut,
		\item for every edge $e \in E$ the sum of angles of the underlying taut structure of $\mathcal{T}$, over all embeddings of $e$ into tetrahedra, equals $2\pi$ \cite[Definition 1.2]{veer_strict-angles}.	
	\end{itemize}
	This implies that triangles attached to an edge $e$ are grouped into two \emph{sides}, separated by a pair of $\pi$ angles at $e$; see Figure \ref{fig:branch_eqn}. We distinguish a pair of the \emph{lowermost} and a pair of the \emph{uppermost} (relative to the coorientation) triangles attached to $e$. In Figure~\ref{fig:branch_eqn} triangles $f_1, f_1'$ are the lowermost and triangles $f_3, f_2'$ are the uppermost.
	
	We say that a tetrahedron  $t \in T$ is \emph{immediately below} $e \in E $ if $e$ is the top diagonal of~$t$, and \emph{immediately above} $e$, if $e$ is the bottom diagonal of~$t$. The remaining $\text{degree}(e)-2$ tetrahedra attached to $e$ are called its \emph{side tetrahedra}. Similarly as with triangles, we distinguish a pair of the \emph{lowermost} and a pair of the \emph{uppermost} side tetrahedra of $e$. 
	
	We denote a triangulation with a transverse taut structure by $(\mathcal{T}, \alpha)$.  Note that if~$M$ has a transverse taut triangulation $(\mathcal{T}, \alpha)$ then it also has a transverse taut triangulation obtained from $(\mathcal{T}, \alpha)$ by reversing coorientations on all faces of~$\mathcal{T}$. We denote this triangulation by $(\mathcal{T}, -\alpha)$. These two triangulations have the same underlying taut structure, with tops and bottoms of tetrahedra swapped. 
	\begin{remark*} Triangulations as described above were introduced by Lackenby in~\cite{Lack_taut}, where they are called  \emph{taut triangulations}. 
	\end{remark*}
	\subsection*{Veering triangulations} 	A \emph{veering tetrahedron} is an oriented taut tetrahedron whose equatorial edges are coloured alternatingly red and blue as shown in Figure~\ref{fig:veering_tetra}. A taut triangulation  is \emph{veering} if a colour (red/blue) is assigned to each edge of the triangulation so that every tetrahedron is veering. The convention is that red colour indicates a right-veering edge, and blue colour indicates a left-veering edge; see the original definition of veering triangulations due to Agol \cite[Definition 4.1]{Agol_veer}.
	\begin{figure}[h]
		\begin{center}
			\includegraphics[width=0.21\textwidth]{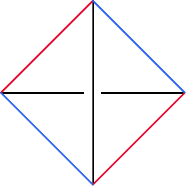}
		\end{center}
		\caption{A veering tetrahedron. The underlying taut structure assigns angle zero to the equatorial edges of the tetrahedron and angle $\pi$ to its diagonal edges.}
		\label{fig:veering_tetra} 
	\end{figure}
	
	In this paper we consider only transverse taut veering triangulations, where we can distinguish top faces from bottom faces. We skip the adjective ``transverse taut" in the remaining of the paper. 
	Every edge of a veering triangulation is of degree at least 4 and has at least two triangles on each of its sides \cite[Definition~4.1]{Agol_veer}. We denote a veering triangulation by $\mathcal{V}= (\mathcal{T}, \alpha, \nu)$, where $\nu$ corresponds to the colouring of edges. If $M$ has a veering triangulation $(\mathcal{T}, \alpha, \nu)$, then it also has a veering triangulation $(\mathcal{T}, -\alpha, \nu)$, where the coorientations on faces are reversed, $(\mathcal{T}, \alpha, -\nu)$, where the colours on edges are interchanged and   $(\mathcal{T},- \alpha, -\nu)$, where both coorientations of faces and colours on edges are interchanged. Note that the operation $\nu \mapsto -\nu$ corresponds to reversing the orientation of $M$.
	
	The following lemma and the subsequent corollary will be used in the proof of Proposition \ref{prop:big_algorithm_correct}, where we prove that algorithm \texttt{LowerVeeringPolynomial} correctly computes the lower veering polynomial. 
	
	\begin{lemma}\label{lemma:same_colour}
		Let $\mathcal{V}= ((T,F, E), \alpha, \nu)$ be a veering triangulation. 
		Then for any $f \in F$ the bottom diagonal of a tetrahedron immediately above~$f$ and the top diagonal of the tetrahedron immediately below $f$ are of the same colour.
	\end{lemma}
	\begin{proof}
		Denote by $b$ the bottom diagonal of the tetrahedron immediately above~$f$ and by $t$  the top diagonal of the tetrahedron immediately below~$f$. Then clearly $b, t$ are distinct edges in the boundary of $f$, otherwise $\mathcal{V}$ would not be veering. Since $b,t$ are diagonal edges, there is another edge of $f$ which has the same colour as $b$ and an edge of $f$ which has the same colour as $t$. Thus $b,t$ cannot be of a different colour.
	\end{proof}
	\begin{definition}
		We say that a triangle of a veering triangulation is \emph{red} (respectively \emph{blue}) if two of its edges are red (respectively blue).
	\end{definition}
	\begin{corollary}\label{cor:triangles_of_same_colour}
		Let $\mathcal{V}= ((T,F, E), \alpha, \nu)$ be a veering triangulation. Among all triangles attached to an edge $e\in E$ there are exactly four which have the same colour as $e$. They are the two uppermost and the two lowermost triangles attached to $e$.
	\end{corollary}
	\begin{proof}
		By Lemma \ref{lemma:same_colour} the lowermost and uppermost triangles attached to~$e$ are of the same colour as $e$.
		Conversely, suppose $f \in F$ is neither a lowermost nor an uppermost triangle around $e$. Then $e$ is an equatorial edge of both the tetrahedron immediately above $f$ and the tetrahedron immediately below $f$. Again by Lemma \ref{lemma:same_colour}  the bottom diagonal of the first and the top diagonal of the latter are of the same colour. Since they are different edges of $f$ (otherwise~$\mathcal{V}$ would not be veering), it follows that $f$ is of a different colour than $e$.
	\end{proof}
	\subsection*{The veering census}
	The data on transverse taut veering structures on ideal triangulations of orientable 3-manifolds consisting of up to 16 tetrahedra is available in the veering census \cite{VeeringCensus}. A veering triangulation in the census is described by a string of the form
	\begin{equation}\label{string}
		\texttt{[isoSig]\underline{ }[taut angle structure]}.
	\end{equation}
	The first part of this string is the isomorphism signature of the triangulation. It identifies a triangulation uniquely up to combinatorial isomorphism. Isomorphism signatures have been introduced in \cite[Section 3]{Burton_isoSig}. The second part of the string records the transverse taut structure, up to a sign. This means that an entry from the veering census determine $(\mathcal{T}, \pm \alpha, \pm \nu)$. The sign of  $\pm \nu$ depends on the sign of $\pm \alpha$ and the orientation of the underlying manifold. 
	
	We use this description whenever we refer to a concrete example of a veering triangulation. Implementations of all algorithms given in this paper take as an input a string of the form \eqref{string}.
	\section{Structures associated to a transverse taut triangulation}\label{sec:structures}
	In this section we recall the definitions of the \emph{horizontal branched surface} \cite[Subsection 2.12]{SchleimSegLinks}, the \emph{boundary track} \cite[Section 2]{explicit_angle} and the \emph{lower} and \emph{upper tracks} associated to a transverse taut triangulation \cite[Definition~4.7]{SchleimSegLinks}. The lower and upper tracks are directly used in the definition of the taut polynomial; see Section \ref{sec:taut_poly}. The boundary track is used in Subsection \ref{subsec:boundary_cycles} to encode boundary components of a surface carried by the transverse taut triangulation.
	\subsection{Horizontal branched surface}\label{subsec:horizontal}
	Let $(\mathcal{T}, \alpha)$ be a transverse taut triangulation of $M$.  Since $\mathcal{T}$ is endowed with a compatible taut structure, we can view the 2-skeleton $\mathcal{T}^{(2)}$ of~$\mathcal{T}$ as a 2-dimensional complex with a well-defined tangent space everywhere, including along its 1-skeleton. Thus ~$\mathcal{T}^{(2)}$ determines a \emph{branched surface} (without vertices) in $M$. It is called the \emph{horizontal branched surface} and denoted by $\mathcal{B}$ \cite[ Subsection 2.12]{SchleimSegLinks}. The branch locus of $\mathcal{B}$ is equal to the 1-skeleton $\mathcal{T}^{(1)}$. In particular, we can see~$\mathcal{B}$ as  \emph{ideally triangulated} by the triangular faces of $\mathcal{T}$. We denote this triangulation of~$\mathcal{B}$ by $(F, E)$. For a more general definition of a branched surface see \cite[p.~532]{Oertel_lam}.

	\subsection*{Branch equations} 
	Let $e \in E$ be an edge of degree $d$ of a transverse taut triangulation~$\mathcal{T}$.  Let $f_{1}, f_{2}, \ldots, f_{k}$ be triangles attached to $e$ on the right side, ordered from the bottom to the top. Let $f'_1, f'_{2}, \ldots, f'_{d-k}$ be triangles attached to $e$ on the left side, also ordered from the bottom to the top. 
	Then~$e$ determines the following relation between the triangles attached to it
	\begin{equation} \label{eqgn:branch_eqn}
		f_{1} + f_{2} + \ldots + f_{k} = f'_1 + f'_2 + \ldots + f'_{d-k}.
	\end{equation}
	\begin{figure}[h]
		\begin{center}
			\includegraphics[width=0.26\textwidth]{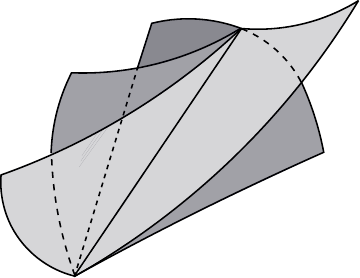} 
			\put(-120,15){$f_3$}
			\put(-106,45){$f_2$}
			\put(-77,75){$f_1$}
			\put(-2,67){$f'_2$}
			\put(-15,25){$f'_1$}
		\end{center}
		\caption{Edge with the branch equation $f_1+f_2+f_3=f'_1+f'_2$.}
		\label{fig:branch_eqn}
	\end{figure}
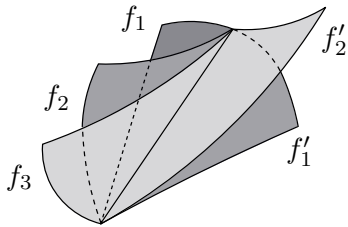
	We call this equation the \emph{branch equation} of $e$. An example is given in Figure \ref{fig:branch_eqn}.
	A transverse taut triangulation with $n$ tetrahedra determines a system of~$n$ branch equations. 	
	In Subsection \ref{subsec:encoding} we consider a matrix
	\[B:\zz^E \rightarrow \zz^F,\]
	which assigns to an edge $e$ its branch equation. 
	$B$ is called the \emph{branch equations matrix} for $(\mathcal{T}, \alpha)$. For $e \in E$ as in \eqref{eqgn:branch_eqn} we have	
	\[B(e) = f_{1} + f_{2} + \ldots + f_{k} -\left( f'_1 + f'_2 + \ldots + f'_{d-k}\right).\]
	\subsection*{Surfaces carried by a transverse taut triangulation} 
	Given  a nonzero, nonnegative, integral solution $w = (w_1, \ldots, w_{2n})$ to the system of branch equations of $(\mathcal{T}, \alpha)$ we can build a surface 
	\[S^w = \sum\limits_{i=1}^{2n}w_i f_i,\]
	properly embedded in $M$.

	We say that a surface $S$ properly embedded in $M$ is \emph{carried} by a transverse taut triangulation $(\mathcal{T}, \alpha)$ if it can be realised as $S^w$ for some nonnegative $w \in \zz^{2n}$. This is equivalent to the definition given in \cite[Subsection 2.14]{SchleimSegLinks}.
	
	If there exists a strictly positive integral solution $w$ we say that $(\mathcal{T}, \alpha)$ is \emph{layered}. In this case $S^w$ is (a multiple of) a fibre of a fibration of $M$ over the circle \cite[Theorem~5.15]{LMT}. If there exists a nonnegative, nonzero integral solution, but no strictly positive integral solution, then we say that $(\mathcal{T}, \alpha)$ is \emph{measurable}.
	
	\subsection{Boundary track}\label{subsec:boundary_track}
	An object which is strictly related to the horizontal branched surface is the \emph{boundary track}. To define it, it is necessary to view the manifold $M$ with a transverse taut triangulation $(\mathcal{T}, \alpha)$ in the compact model. 	
	More information on the boundary track can be found in  \cite[Section 2]{explicit_angle}.
	\begin{definition}
		Let $(\mathcal{T}, \alpha)$ be a (truncated) transverse taut triangulation of a (compact) 3-manifold~$M$. Denote by $\mathcal{B}$ the horizontal branched surface for  $(\mathcal{T}, \alpha)$. The \emph{boundary track} $\beta$ of $(\mathcal{T}, \alpha)$  is the intersection $\mathcal{B} \cap \partial M$.
	\end{definition}
	In Figure \ref{fig:boundary_train_track} we present a local picture of a boundary track around one of its switches.  A global picture of the boundary track for the veering triangulation \texttt{cPcbbbiht\_12} of the figure eight knot complement is presented in Figure~\ref{fig:boundary_torus_m004}.  
	\begin{figure}[h]
		\begin{center}
			\includegraphics[width=0.21\textwidth]{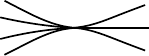}
		\end{center}
		\caption{The boundary  track around one of its switches. This switch corresponds to an endpoint of an edge of degree 7.}
		\label{fig:boundary_train_track} 
	\end{figure}
	
	Each edge of $\mathcal{T}$ has two endpoints. Therefore for every $e \in E$ the boundary track $\beta$ has two switches of the same degree that can be labelled with~$e$. Each triangle $f \in F$ has three arcs around its corners; see Figure~\ref{fig:cooriented_triangle}. These corner arcs are in a bijective correspondence with the branches of $\beta$. Therefore for every $f \in F$ the track $\beta$ has three branches that we label with~$f$.	
	If~$M$ has $b \geq 1$ boundary components $T_1, \ldots, T_b$, then~$\beta$ is a disjoint union of train tracks $\beta_1, \ldots, \beta_b$ in boundary tori $T_1, \ldots, T_b$, respectively. 

	The boundary track $\beta$ of $(\mathcal{T}, \alpha)$ is transversely oriented by $\alpha$. We  orient the branches of $\beta$ using the right hand rule and the coorientation on $f$; see Figure~\ref{fig:cooriented_triangle}. Therefore every switch has a collection of \emph{incoming} branches and a collection of \emph{outgoing branches}. Moreover, branches within these collections can be ordered from bottom to top. In particular, for every branch $\epsilon$ of $\beta$ we  can consider 
	\begin{itemize}
		\item branches outgoing from the initial switch of $\epsilon$ above $\epsilon$,
		\item branches incoming to the terminal switch of $\epsilon$ above $\epsilon$.
	\end{itemize}
	We use these observations in Subsection \ref{subsec:boundary_cycles} where we give algorithm \texttt{BoundaryCycles}.
	\begin{figure}[h]
		\begin{center}
			\includegraphics[width=0.23\textwidth]{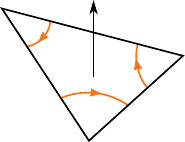}
		\end{center}
		\caption{Coorientation on a triangle of $\mathcal{T}$ determines orientation on the branches of $\beta$ by the right hand rule. Coorientation on the branches of $\beta$ (not indicated in the picture) agrees with the coorientation on the triangle in which they are embedded.}
		\label{fig:cooriented_triangle} 
	\end{figure}
	\subsection{Train tracks in the horizontal branched surface}\label{subsec:dual_tracks}
	In the previous subsection we considered the boundary track associated to a transverse taut triangulation $(\mathcal{T}, \alpha)$. In this subsection we consider an entirely different kind of train tracks associated to $(\mathcal{T}, \alpha)$, called \emph{dual train tracks}. They are  embedded in the horizontal branched surface~$\mathcal{B}$ and are dual to its triangulation $(F, E)$. A good reference for train tracks in surfaces is \cite{Penner_tt}. We need to modify the standard definition of a train track so that it is applicable to our setting.
	
	We construct train tracks in $\mathcal{B}$ dual to  the triangulation $(F,E)$ by gluing together ``ordinary" train tracks in individual triangles of that triangulation. We restrict the class of train tracks that we allow in those triangles. The train tracks that we allow are called \emph{triangular}.
	\begin{definition}\label{defn:triangle_train_track}
		Let $f$ be an ideal triangle. By a \emph{triangular train track} in~$f$ we mean a graph $\tau_f \subset f$ with four vertices and three edges, such that\begin{itemize}
			\item one vertex $v$ is in the interior of $f$ and the remaining three vertices are at the midpoints of the three edges in the boundary of $f$, one for each edge,
			\item for each vertex $v'$ different than $v$ there is an edge joining $v$ and $v'$,
			\item all edges are $C^1$-embedded,
			\item there is a well-defined tangent line to $\tau_f$ at $v$.
		\end{itemize}
		See Figure \ref{fig:triangle_train_track}. We call the vertex $v$  in the interior of $f$ a \emph{switch} of~$\tau_f$.  The edges of~$\tau_f$ are called \emph{half-branches}. Each half-branch has one \emph{switch endpoint} and one \emph{edge endpoint}.
	\end{definition}
	\begin{figure}[h]
		\begin{center}
			\includegraphics[width=0.15\textwidth]{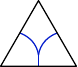}
			\put(-42,-10){large}
			\put(-75,26){small}
			\put(-11,26){small}
		\end{center}
		\caption{A triangular train track.}
		\label{fig:triangle_train_track}
	\end{figure}
	 A tangent line to $\tau_f$ at a switch $v$ distinguishes two \emph{sides} of $v$. Two half-branches  are on different sides of $v$ if an only if the path contained in $\tau_f$ which joins their edge endpoints is smooth. A switch~$v$ has one half-branch on one side and two on the other. We call the half-branch which is the unique half-branch on one side of $v$ the \emph{large half-branch} of $\tau_f$. The remaining two half-branches are called \emph{small half-branches} of~$\tau_f$.	
	The switch of $\tau_f$ determines a relation of the form 
	\begin{equation} e_0=e_1+e_2, \label{eq:a=b+c}\end{equation}
	between the three half-branches of $\tau_f$, 
	where $e_0$ denotes the large half-branch of $\tau_f$.
	We call this relation a \emph{switch equation} of $\tau_f$. 
	
	\begin{definition}
		A \emph{dual train track} in $(\mathcal{B}, F)$ is a finite graph $\tau\subset \mathcal{B}$ whose restriction to any ideal triangle $f$ of the ideal triangulation of $\mathcal{B}$ by $(F, E)$ is a triangular train track, which we denote by $\tau_f$. Every \emph{switch}\slash\emph{half-branch} of~$\tau$ is a \emph{switch}\slash \emph{half-branch} of $\tau_f$ for some $f \in~F$, respectively.  
	\end{definition}  
	Switch equations of triangular train tracks $\tau_f$ give rise to \emph{switch equations} of $\tau$ via identifying half-branches which share the same edge midpoint. Hence switch equations of $\tau \subset \mathcal{B}$ are relations between \emph{edges} of the triangulation $(\mathcal{T}, \alpha)$.
	\begin{definition}\label{defn:switch_equation}
		Let $\tau$ be a dual train track in $(\mathcal{B}, F)$ and let $f \in F$. Suppose \mbox{$e_0, e_1, e_2 \in E$} are embedded in the boundary of $f$ and that $e_0$ contains the edge endpoint of the large half-branch of $\tau_f$. 	
		We call the following relation between edges of~$\mathcal{T}$ 
		\[e_0=e_1+e_2\]
		a \emph{switch equation} of $\tau$ in $f$.
	\end{definition}

	\subsection*{The lower and upper tracks of a transverse taut triangulation}
	The transverse taut structure on a triangulation endows its horizontal branched surface $\mathcal{B}$ with a pair of dual train tracks which we call, following \cite[Definition 4.7]{SchleimSegLinks}, the \emph{lower} and \emph{upper}  tracks of~$\mathcal{B}$. 
	
	\begin{definition}Let $(\mathcal{T}, \alpha)$ be a transverse taut triangulation. Let $\mathcal{B}$ be the horizontal branched surface of $(\mathcal{T}, \alpha)$ equipped with the ideal triangulation $(F, E)$ determined by~$\mathcal{T}$. 
		
		The \emph{lower track} $\tau^L$ of $\mathcal{T}$ is the dual train track in $\mathcal{B}$ such that for every $f \in F$ the large-half branch of $\tau^L_f$ is dual to this edge of $f$ which is the top diagonal of the tetrahedron of $\mathcal{T}$ immediately below~$f$.
		
		The \emph{upper track} $\tau^U$ of $\mathcal{T}$ is the dual train track in $\mathcal{B}$ such that for every $f \in F$ the large-half branch of $\tau^U_f$ is dual to this edge of $f$ which is the bottom diagonal of the tetrahedron of $\mathcal{T}$ immediately above~$f$.
	\end{definition}
	We introduce the following names for the edges of $f \in F$ which are dual to large half-branches of $\tau_f^L$ or $\tau_f^U$.
	\begin{definition}
		Let $(\mathcal{T}, \alpha)$ be a transverse taut triangulation. We say that an edge in the boundary of $f \in F$ is the \emph{lower large} (respectively the \emph{upper large}) edge of $f$ if it contains the edge endpoint of the large half-branch of $\tau_{f}^L$ (respectively $\tau_{f}^U$).
	\end{definition}
	To define the lower and upper tracks we do not need a veering structure on the triangulation. However, in the case of veering triangulations we can figure out the lower and upper tracks restricted to the faces of a given tetrahedron $t$ without looking at the tetrahedra adjacent to $t$. Instead, the tracks are encoded by the colours of the edges of $t$; see Figure \ref{fig:lower_tt}. A more precise statement appears in the following lemma, which can be deduced from Lemma 3.2 of \cite{LMT}. We use it in the proof of Lemma \ref{lemma:linear_dependence}.
		\begin{figure}[h]
		\begin{center}
			\includegraphics[width=0.35\textwidth]{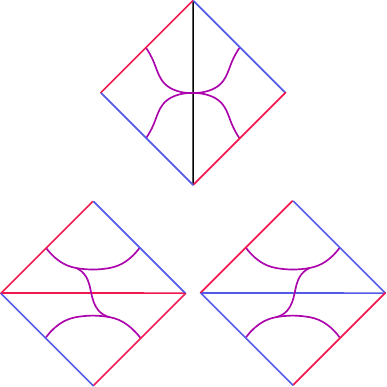} 
			\includegraphics[width=0.35\textwidth]{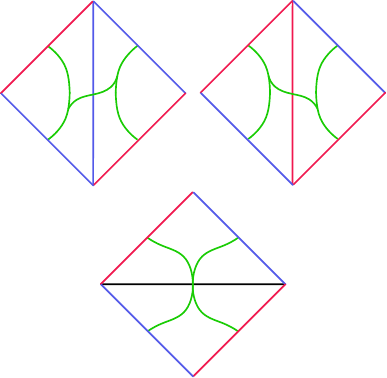} 
			\put(-300,5){(a)}
			\put(-140,5){(b)}
		\end{center}
		\caption{Squares in the top row represent top faces of a tetrahedron and squares in the bottom row represent bottom faces of a tetrahedron. (a) The lower track in a veering tetrahedron. In the bottom faces there are two options, depending on the colour of the bottom diagonal. 
			(b) The upper track in a veering tetrahedron. In the top faces there are two options, depending on the colour of the top diagonal.} 
		\label{fig:lower_tt}
	\end{figure}
	\begin{lemma}\label{lem:large_edges}
		Let $\mathcal{V} = (\mathcal{T}, \alpha, \nu)$ be a veering triangulation. Let $t$ be one of its tetrahedra. The lower large edges of the bottom faces of $t$ are the equatorial edges of $t$ which are of the same colour as the bottom diagonal of~$t$. The upper large edges of the top faces of $t$ are the equatorial edges of $t$ which are of the same colour as the top diagonal of~$t$. \qed
	\end{lemma}  
	The pictures of the lower and upper tracks in a veering tetrahedron are presented in Figure~\ref{fig:lower_tt}(a) and Figure~\ref{fig:lower_tt}(b), respectively.

	\begin{remark}\label{remark:changes_on_tracks}
		The operation $\nu \mapsto -\nu$ does not affect the lower and upper tracks as their definition does not depend on the 2-colouring on the veering triangulation. 
		The operation $\alpha \mapsto -\alpha$ interchanges the lower and upper track.
	\end{remark}
	
	\section{The maximal free abelian cover} \label{sec:Mab}
	The aim of this paper is to give algorithms to compute the taut, the veering and the Teichm\"uller polynomials. All of them are related to the \emph{maximal free abelian cover} $\Mab$ of  $M$.
	This covering space corresponds to the kernel of the homomorphism
	\[\pi_1(M)\rightarrow \bigslant{H_1(M;\zz)}{\text{torsion}}.\] 
	The deck group of the covering $\Mab \rightarrow M$ is isomorphic to
	\[H=\bigslant{H_1(M;\zz)}{\text{torsion}}.\]

	Let $r$ denote the rank of $H$. The integral group ring  on $H$  is isomorphic to the ring $\zz\lbrack u_1^{\pm 1}, \ldots, u_r^{\pm 1}\rbrack$ of Laurent polynomials. We denote it by $\zz \lbrack H \rbrack$. The multiplicative subgroup of Laurent monomials (with coefficient 1) in $\zz \lbrack H \rbrack$ is isomorphic to~$H$. If a basis $(b_1,\ldots, b_r)$ of~$H$ is fixed then we choose the isomorphism~$\Psi$ to be $b_i \mapsto u_i$.  For $u=(u_1, \ldots, u_r)$ the tuple of variables and $v =(v_1,\ldots,v_r)\in \zz^r$  by $u^v$ we denote the monomial $u_1^{v_1}\cdots u_r^{v_r}$. 	
	
	Suppose $M$ is equipped with a transverse taut triangulation $(\mathcal{T}, \alpha)$. Then a free abelian cover $\Mab$ admits a triangulation $\Tab$ induced by $\mathcal{T}$ via the covering map $\Mab\rightarrow M$.  It is also transverse taut, as  coorientations on triangular faces can be lifted from $\mathcal{T}$. If $\mathcal{T}$ is additionally veering, then so is $\Tab$. In this case we typically use the notation $\mathcal{V}$ for the triangulation of $M$ and $\Vab$ for the triangulation of $\Mab$. We orient ideal simplices of $\Tab$ in such a way that the restriction of $\Mab\rightarrow M$ to each simplex is orientation-preserving. 
	
	Ideal tetrahedra, triangles and edges of $\Tab$ are lifts of elements of $T, F , E$, respectively. They can be indexed by elements of $H$, hence we denote their sets by $H\cdot T, H\cdot F, H\cdot E$. The free abelian groups generated by $H\cdot T$, $H\cdot F$ and  $H\cdot E$ all admit a $\zz\lbrack H\rbrack$-module structure, via the obvious action of $H$ on them. Therefore we identify these groups with the free $\zz\lbrack H\rbrack$-modules $\zz\lbrack H\rbrack^T, \zz\lbrack H\rbrack^F, \zz\lbrack H\rbrack^E$, respectively.
	
	The lower and upper tracks of $\mathcal{T}$ induce the lower and upper tracks of~$\Tab$. We denote them by $\lowert$ and $\upert$, respectively. 
	
	\subsection{Labelling ideal simplices of  $\Tab$}\label{subsec:labelling}
	Let $(\mathcal{T}, \alpha)$ be a transverse taut triangulation of~$M$. By $(\Tab, \alpha^{ab})$ we denote the induced triangulation of~$\Mab$. Let $ \mathcal{F}$ be a connected fundamental domain for the action of $H$ on~$\Mab$ whose closure $ \overline{\mathcal{F}}$ is triangulated by lifts of tetrahedra $t\in T$. 
	We label the tetrahedra in $ \overline{\mathcal{F}}$ with $1\in H$. Every other tetrahedron in the  triangulation $\Tab$ is a translate of some $1\cdot t$ by an element $h\in H$, hence we denote it by $h\cdot t$. We also say that $h\cdot t$ has \emph{$H$-coefficient}~$h$. 
	
	For the remaining of the paper, unless stated otherwise, the labelling for the lower dimensional ideal simplices is as follows: 
	\begin{itemize}
		\item triangle $h\cdot f$, for $f \in F$, $h\in H$, is a top triangle of a tetrahedron with $H$-coefficient $h$,
		\item edge $h\cdot e$, for $e\in E$, $h\in H$, is a top diagonal of a tetrahedron with $H$-coefficient~$h$.
	\end{itemize}
	In other words, triangles and edges inherit their $H$-coefficient from the unique tetrahedron immediately below them. Once a basis for $H$ is fixed, we replace the $H$-coefficients of simplices with their \emph{Laurent coefficients}.	They are the images of $H$-coefficients  under the isomorphism~$\Psi$.
	
	\begin{remark*}
		Throughout the paper we use the multiplicative convention for~$H$.
	\end{remark*}
	Every triangle of $\Tab$ is a top triangle of  $h_1\cdot t_1$ and a bottom triangle of $h_2\cdot t_2$ for some $h_1, h_2 \in H$, $t_1,t_2 \in T$. Moreover, for all lifts of a triangle $f \in F$ the corresponding  product $h_2h_1^{-1}$ is the same. This motivates the following definition.
	\begin{definition}\label{defn:H-pairings}
		The	\emph{$H$-pairings} for $(\mathcal{T},\mathcal{F})$ are elements $h_i \in H$ associated to triangles $f_i \in F$ 	
		such that the tetrahedron immediately above $h\cdot f_i$ is in $h_ih\cdot \mathcal{F}$. We also say that $h_i$ is the \emph{$H$-pairing} of $f_i$ relative to $\mathcal{F}$. 
	\end{definition}
	A transverse taut triangulation $\Tab$ together with a fixed fundamental domain~$\mathcal{F}$  determine the $H$-pairings and, once a basis of $H$ is fixed, the \emph{face Laurents} --- their images under the isomorphism $\Psi$. 
	We, however, need to reverse this process. Namely, given a triangulation $\mathcal{T}$ of~$M$, we compute a tuple $\eta$ of $2n$ $H$-pairings,  which determines a consistent labelling of ideal simplices of $\Tab$ by elements of $H$. In this way we encode the whole infinite triangulation $\Tab$ with a pair of finite objects $(\mathcal{T}, \eta)$. This procedure is a subject of the next subsection.

	\subsection{Encoding the triangulation $\Tab$ by $(\mathcal{T},\eta)$}\label{subsec:encoding}
	We want to encode the triangulation $\Tab$ by $\mathcal{T}$ and a finite tuple of $H$-pairings associated to the triangles of $\mathcal{T}$. The latter depends on the chosen fundamental domain~$\mathcal{F}$ for the action of $H$ on $\Tab$. We fix it using the dual graph of $\mathcal{T}$.
	
	\subsection*{The dual 2-complex and the dual graph}
	Let $\mathcal{T}=(T,F, E)$ be an ideal triangulation. We use $\mathcal{D}$ to denote its dual complex. It is a 2-dimensional CW-complex: it has $n$ vertices, each corresponding to some $t\in T$, $2n$ edges, each corresponding to a triangular face $f \in F $, and~$n$ two-cells, each corresponding to an edge $e \in E$. 
	
	By $X$ we denote the 1-skeleton of $\mathcal{D}$. We call $X$ the \emph{dual graph} of~$\mathcal{T}$. Whenever~$\mathcal{T}$ is transverse taut, we assume that $X$ is endowed with the ``upward" orientation on edges, coming from the coorientation on the triangular faces of $\mathcal{T}$. 
	\subsection*{Fixing a fundamental domain}
	Suppose $\mathcal{T}$ is transverse taut. We orient the edges of the dual graph $X$ consistently with the transverse taut structure on $\mathcal{T}$.
	Let $Y$  be a spanning tree of $X$. Then $Y$ has $n$ vertices $t \in T$ and $n-1$ edges. 
	
	Since~$\mathcal{D}$ is a deformation retract of $M$, it has a free abelian cover~$\Dab$ with the deck transformation group isomorphic to $H$. Let $\Xab$ be the preimage of~$X$ under the covering map $\Dab \rightarrow \mathcal{D}$. Fix a lift~$\widetilde{Y}$ of $Y$ to $\Xab$. 
	The lift $\widetilde{Y}$ determines a fundamental domain $\mathcal{F}$ for the action of $H$ on $\Tab$  built from:
	\begin{itemize}
		\item the interiors of all tetrahedra of $\Tab$ dual to vertices of $\widetilde{Y}$,
		\item top diagonal edges of tetrahedra of $\Tab$ dual to vertices of $\widetilde{Y}$,
		\item the interiors of triangles of $\Tab$ dual to the edges  of $\Xab$ which join two vertices of $\widetilde{Y}$,
		\item the interiors of triangles of $\Tab$ dual to the edges of  $\Xab$ which run from a vertex of $\widetilde{Y}$ to a vertex not in $\widetilde{Y}$.
	\end{itemize}
	\begin{definition}\label{defn:fund_dom_tree}
		We say that $\mathcal{F}$ constructed as above is the  (upwardly closed) fundamental domain for the action of $H$ on $\Tab$ \emph{determined by the spanning tree} $Y$ of $X$. 
	\end{definition}
	The fundamental domain $\mathcal{F}$ is well-defined up to a translation by an element of $H$. 

	\subsection*{Finding $H$-pairings}
	A choice of the spanning tree $Y$ of the dual graph~$X$ not only determines the fundamental domain $\mathcal{F}$ for the action of $H$ on $\Tab$, but also gives an easy way to find a presentation for $\pi_1(\mathcal{D})$, and hence for~$H$, in terms of elements of~$F$. 
	
	We will be more general and consider a free abelian  quotient $H^C$ of $H_1(M;\zz)$. This generalisation will be used only in Section \ref{sec:Teich}.
	The group $H^C$ will be given to us in the following way. Let~$Y$  be a spanning tree of $X$. Let $F_Y$ be the subset of~$F$ consisting of triangles dual to the edges \emph{not} in $Y$. We call the elements of $F_Y$ the \emph{non-tree edges}, and elements of $F \bez F_Y$ -- the \emph{tree edges}.
	Recall the branch equations matrix 
	\[B:\zz^E\rightarrow \zz^F\]
	associated to a transverse taut triangulation  $(\mathcal{T},\alpha)$ of~$M$; see Subsection \ref{subsec:horizontal}. Let
	\[B_Y:\zz^E \rightarrow \zz^{F_Y}\] be obtained from~$B$ by deleting the rows corresponding to the tree edges.
	
	Let $\mathcal{D}_Y$ denote the 2-complex obtained from $\mathcal{D}$ by contracting $Y$ to a point. Then~$B_Y$ is the boundary map from the 2-chains to the 1-chains of $\mathcal{D}_Y$. Thus $H_1(\mathcal{D}_Y;\zz)$ is isomorphic to the cokernel of $B_Y$. Moreover
	\[H_1(M;\zz) \cong H_1(\mathcal{D};\zz)\cong H_1(\mathcal{D}_Y;\zz),\]
	where the first isomorphism follows from the fact that $\mathcal{D}$ is a deformation retract of~$M$ and the second --- from the fact that $\mathcal{D}_Y$ is homotopy equivalent to $\mathcal{D}$.
	
	Therefore 
	$H_1(M;\zz)$ is generated by $n+1$ non-tree edges $f_1, \ldots, f_{n+1}$ which satisfy~$n$ relations $r_1,\ldots, r_n$. 
	We add a finite collection~$C$ of additional relations $c_1, \ldots, c_k$. Each of them corresponds to a 1-cycle in $\mathcal{D}_Y$. This determines a group \[G=\langle f_1, \ldots, f_{n+1} \ | \ r_1, \ldots, r_n, c_1, \ldots, c_k \rangle.\] 
	Let $(B|C)_Y$ be the augmentation of the matrix $B_Y$ by the columns $c_1, \ldots, c_k$. The group $G$ is isomorphic to the cokernel of $(B|C)_Y$. 
	We set 
	\[H^C =  \bigslant{G}{\text{torsion}}.\]
	\begin{remark*}
		$H^C$-pairings encode a triangulation $\mathcal{T}^C$ of a free abelian cover of $M$ with the deck group isomorphic to $H^C$. We use them only in algorithm \texttt{Teichm\"ullerPolynomial}; see Subsection \ref{subsec:Teich_comp}.
	\end{remark*}
		Suppose that $H^C$ is of rank $r$. Let 
		\begin{equation}\label{eqn:smith}
			S= U(B|C)_Y V
		\end{equation} be the Smith normal form of $(B|C)_Y$. 		 
	  Let us denote the elements of $F_Y$ by $f_1, \ldots, f_{n+1}$. The matrix $U$ transforms the basis $(f_1,\ldots,f_{n+1})$ of $\zz^{F_Y}$ to another basis $(\mu_1,\ldots,\mu_{n+1})$. 
	  The last $r$ rows of both $S$ and $U(B|C)_Y$ are zero. 
	  In particular, $(\mu_{n-r+2},\ldots, \mu_{n+1})$ is a basis for $H^C$. The coefficients of $\mu_i$ as a linear combination of $\lbrace f_1,\ldots,f_{n+1} \rbrace$ are equal to the consecutive entries of the $i$-th column of $U^{-1}$. Therefore the last $r$ columns of the matrix $U^{-1}$ give a representation of the basis elements of  $H^C$ as 1-cycles in $\mathcal{D}_Y$.
		
		The consecutive entries of the $i$-th column of $U$ gives us the coefficients of $f_i$ as a linear combination of $(\mu_1, \ldots, \mu_{n+1})$. Since $\mu_1, \ldots, \mu_{n-r+1}$ are 0 in $H^C$ it follows that the last $r$ entries of the $i$-th column of $U$ correspond to the $H^C$-pairing of $f_i$ written with respect to the basis $(\mu_{n-r+2},\ldots, \mu_{n+1})$ of $H^C$.
	\begin{remark*}
		The $H^C$-pairings for $f \in F \bez F_Y$ (tree edges) are all trivial.
	\end{remark*}
	
	\subsection{Algorithm \texttt{FacePairings}}\label{subsec:alg_encoding}
	
	Recall that an ideal triangulation $\mathcal{T} = (T,F, E)$ of a 3-manifold determines the dual graph $X$ with $|T|$ vertices and $2\cdot |T|$ edges. In the following algorithm by $\texttt{SpanningTree}$ we denote an algorithm which takes as an input an ideal  triangulation $\mathcal{T} = (T,F, E)$ and returns a subset of $F$ consisting of triangles dual to the edges of a spanning tree~$Y$ of $X$. 
	
	\begin{algorithm}[H]
		\caption*{\textbf{Algorithm} \texttt{FacePairings} \\ Encoding the triangulation a free abelian cover}\label{alg:fund_domain}
		\begin{algorithmic}[1]
			\Statex 	\textbf{Input:} \begin{itemize}\item A transverse taut triangulation $(\mathcal{T},\alpha)$ of a cusped 3-manifold $M$ with $n$ ideal tetrahedra, $\mathcal{T} = (T,F,E)$
				\item A list $C$ of 1-cycles in the dual graph $X$
				\item Optional: return type = ``matrix''
			\end{itemize}
			\Statex 	\textbf{Output:}\begin{itemize}\item Default: a tuple  of $2n$ face Laurents encoding the triangulation $\mathcal{T}^C$ of a free abelian cover of $M$ with the deck group isomorphic to $H^C$
				\item If return type = ``matrix'': a pair $(U,r)$ where $r$ is the rank of $H^C$ and $U$ is as in \eqref{eqn:smith}
				\end{itemize}
			\State $B:=$  the branch equations matrix of $(\mathcal{T}, \alpha)$ \Comment{$2n\times n$ integer matrix}
			\State $B:=B$.AddColumns($C$)
			\State $Y := \texttt{SpanningTree}(\mathcal{T})$ 
			\State NonTree := $F - Y$
			\State $B:=B.$DeleteRows$(Y)$  \Comment{$(n+1)\times n$ integer matrix}
			\State $S,U,V:=\mathrm{SmithNormalForm}(B)$ \Comment{$S=UBV$}
			\State $r:=$ the number of zero rows of $S$ 
			\If{return type = ``matrix''}
			\State 	\Return $U, r$
			\EndIf
			\State $\eta :=$ the zero matrix with $r$ rows and columns indexed by elements of~$F$
			\For{$f$ in NonTree}	
			\State column $\eta(f) :=$ the last $r$ entries of the column $U(f)$
			\EndFor
			\State FaceLaurents$: = $ the tuple of zero Laurent polynomials, indexed by $F$
			\For{$f$ in $F$}
			\State FaceLaurents$(f):=u^{\eta(f)}$ \Comment{$u=(u_1, \ldots, u_r)$ and $u^v = u_1^{v_1}\cdots u_r^{v_r}$}
			\EndFor
			\State	\Return FaceLaurents
		\end{algorithmic}
	\end{algorithm}

	\FloatBarrier
	\subsection{Polynomial invariants of finitely presented $\ZH$-modules}\label{subsec:fitting}
	Both the taut and veering polynomials are derived from Fitting ideals of certain $\ZH$-modules associated to a 3-manifold~$M$. For this reason in this section we recall definitions of Fitting ideals and their invariants.
	
	Let $\mathcal{M}$ be a finitely presented module over $\zz \lbrack H \rbrack$. Then there exist integers $k, l \in \mathbb{N}$ and an exact sequence
	\[\zz\lbrack H \rbrack^k \overset{A}{\longrightarrow} \zz\lbrack H \rbrack^l \longrightarrow \mathcal{M} \longrightarrow  0\]
	of $\zz\lbrack H \rbrack$-homomorphisms called a \emph{free presentation} of $\mathcal{M}$. The matrix of $A$, written with respect to any bases of $\zz\lbrack H \rbrack^k$ and $\zz\lbrack H \rbrack^l$, is called a \emph{presentation matrix} for $\mathcal{M}$.
	\begin{definition} \cite[Section 3.1]{Northcott}
		Let $\mathcal{M}$ be a finitely presented $\ZH$-module with a presentation matrix $A$ of dimension $l\times k$. We define the $i$-th \emph{Fitting ideal} $\Fit_i(\mathcal{M})$ of~$\mathcal{M}$ to be the ideal in $\ZH$ generated by $(l-i)\times(l-i)$ minors of $A$. 
	\end{definition}
	\begin{remark*}
		Fitting ideals are called \emph{determinantal ideals} in \cite{Traldi} and \emph{elementary ideals} in \cite[Chapter VIII]{crow_fox}.
	\end{remark*}
	\begin{definition}
		Let $\mathcal{M}$ be a finitely presented $\ZH$-module. We define the $i$-th \emph{Fitting invariant} $\delta_i(\mathcal{M})$ of $\mathcal{M}$ to be the greatest common divisor of elements of $\Fit_i(\mathcal{M})$. When $\Fit_i(\mathcal{M}) = (0)$ we set $\delta_i(\mathcal{M})$ to be equal to 0.
	\end{definition}
	Note that Fitting invariants are well-defined only up to a unit in $\ZH$.
	\FloatBarrier
	\section{The taut polynomial} \label{sec:taut_poly}	
	Let $\mathcal{V}= ((T,F, E), \alpha, \nu)$ be a veering triangulation of a 3-manifold $M$. Recall that
	\[H= \bigslant{H_1(M;\zz)}{\text{torsion}}.\]
	Following \cite{LMT} we define the \emph{lower taut module} $\mathcal{E}^L_\alpha(\mathcal{V})$ of $\mathcal{V}$ by the presentation
	\begin{equation}\label{pres:taut} 
		\zz\lbrack H \rbrack^F \overset{D^L}{\longrightarrow} \zz\lbrack H \rbrack^E \longrightarrow  \module^L_\alpha(\mathcal{V}) \longrightarrow 0,
	\end{equation}
	where $D^L$ assigns to a triangle $1\cdot f\in H\cdot F$ the switch equation of $\lowert$ in $1\cdot f$ (recall Definition \ref{defn:switch_equation}). 	
	In other words, suppose face $f$ of $\mathcal{V}$ has  edges  $e_0, e_1, e_2 \in E$ in its boundary. Let $e_0$ denote the top diagonal of the tetrahedron immediately below $f$. Then its lift to $1\cdot f$ has the $H$-coefficient equal to 1. The lifts of the remaining edges $e_1, e_2$ of $f$ to $1\cdot f$ have $H$-coefficients $h_1, h_2 \in H$, respectively. We set
	\[D^L(1\cdot f) = 1\cdot e_0 - h_1\cdot e_1 - h_2 \cdot e_2.\]
	The \emph{lower taut polynomial} of $\mathcal{V}$, denoted by $\Theta_\mathcal{V}^L$, is the zeroth Fitting invariant of $\module^L_\alpha(\mathcal{V})$, that is
	\[\Theta_\mathcal{V}^L = \gcd \Big\lbrace n\times n \text{ minors of } D^L \Big\rbrace.\]
	
	We  can analogously define the \emph{upper taut module} $\module^U_\alpha(\mathcal{V})$ with the presentation matrix~$D^U$ which assigns to $1\cdot f\in H\cdot F$ the switch equation of~$\upert$ in $1\cdot f$. Then the \emph{upper taut polynomial} $\Theta_\mathcal{V}^U$ is the greatest common divisor of the maximal minors of~$D^U$.	
	
	\begin{remark*} The subscript $\alpha$ in $\module^L_\alpha$, $\module^U_\alpha$ reflects the fact that these modules depend only on the transverse taut structure $\alpha$ on $\mathcal{V}$, and not on the colouring~$\nu$. 	
		The reason why we consider only the taut polynomials of veering triangulations is that 	by \cite[Theorem 5.12]{LMT}  a veering triangulation $\mathcal{V}$ of~$M$ determines a unique (not necessarily top-dimensional, potentially empty) face of the Thurston norm ball in $H_2(M, \partial M;\rr)$. Hence its taut polynomials can be seen as invariants of this face. In fact, in Proposition~\ref{prop:lower=upper} we prove that the lower and upper taut polynomials of a veering triangulation are equal, so we get only one invariant.
		
		In contrast, transverse taut triangulations are very common and not canonical in any sense \cite[Theorem 1]{Lack_taut}. Moreover, many results about the taut polynomial that we prove here rely on the veering structure. For example, in Proposition \ref{prop:lower=upper} we use Lemma 2.1 and in Lemma \ref{lemma:linear_dependence} we use Lemma \ref{lem:large_edges}.
	\end{remark*}
	\begin{proposition}\label{prop:lower=upper}
		Let $\mathcal{V} = (\mathcal{T}, \alpha, \nu)$ be a veering triangulation. The lower taut module of $\mathcal{V}$ is isomorphic to the upper taut module of $\mathcal{V}$. Hence
		\[\Theta_{\mathcal{V}}^L = \Theta_\mathcal{V}^U\]
		up to a unit in $\ZH$.
	\end{proposition}
	\begin{proof}	
		Let $1\cdot f \in H\cdot F$ be red. Then the tetrahedron immediately below $1\cdot f$ has a red top diagonal $t$ and the tetrahedron immediately above $1\cdot f$ has a red bottom diagonal~$r$, for some  $t,r \in H\cdot E$; see Lemma \ref{lemma:same_colour}.
		We have
		\begin{gather*}
			D^L(1\cdot f) = t-r-l\\D^U(1\cdot f) = r-t-l
		\end{gather*}
		for some $l\in H\cdot E$, so the signs of the two red edges of $f$ are interchanged. A similar statement is true for blue triangles: the images od $D^L$ and $D^U$ on them differ by swapping the signs of the two blue edges.
		
		If we multiply all columns of $D^L$ corresponding to red triangles of $\mathcal{V}$ by -1, and all rows corresponding to blue edges by $-1$, we obtain the matrix~$D^U$. Hence the maximal minors of $D^L$ and $D^U$ differ at most by a sign.
	\end{proof}
	Thus from now on we only write about the \emph{taut polynomial} $\Theta_\mathcal{V}$ and the taut module~$\module_\alpha(\mathcal{V})$. Throughout this section we use only the lower track.
	\begin{corollary}\label{cor:swaps_equal}
		The taut polynomials of $(\mathcal{T}, \alpha, \nu)$, $(\mathcal{T}, -\alpha, \nu)$, $(\mathcal{T}, \alpha, -\nu)$ and $(\mathcal{T}, -\alpha, -\nu)$ are equal.
	\end{corollary}
	\begin{proof}
		This follows from Proposition \ref{prop:lower=upper} and Remark \ref{remark:changes_on_tracks}. 
	\end{proof}
	
	\subsection{Reducing the number of relations}
	The original definition of the taut polynomial requires computing ${2n \choose n}> 2^n$ minors of $D^L$, which is an obstacle for efficient computation. However, the relations satisfied by the generators of the taut module are not linearly independent. In this subsection we give a recipe to systematically eliminate $n-1$ relations.
	
	The following lemma follows from \cite[Lemma 3.2]{LMT}. We include its proof, because it is important in Proposition \ref{prop:non-tree_generate}.
	
	\begin{lemma}\label{lemma:linear_dependence}
		Let $\mathcal{V} = (\mathcal{T}, \alpha, \nu)$ be a veering triangulation. Each tetrahedron $t \in T$ induces a linear dependence between the columns of $D^L$ corresponding to the triangles in the boundary of $t$.
	\end{lemma}
	\begin{proof}
		Suppose that $1\cdot t$ has red equatorial edges $r_1, r_2$, blue equatorial edges $l_1, l_2$, bottom diagonal $d_b$ and top diagonal $d_t$, where  $r_1, r_2, l_1, l_2,$ $d_b, d_t \in H\cdot E$; see Figure \ref{fig:veering_tetra_with_edges}.
		
		\begin{figure}[h]
			\begin{center}
				\includegraphics[width=0.215\textwidth]{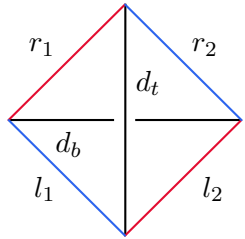}
				\put(-80,70){$r_1$}
				\put(-78,15){$l_1$}
				\put(-19,70){$r_2$}
				\put(-15,15){$l_2$}
				\put(-40,55){$d_t$}
				\put(-70,32){$d_b$}
			\end{center}
			\caption{Edges of the tetrahedron $1\cdot t$ of $\mathcal{V}^{ab}$. }
			\label{fig:veering_tetra_with_edges} 
		\end{figure}
		
		Let $f_1$, $f_2 \in H\cdot F$ be two top triangles of $1\cdot t$ such that
		\begin{gather*}
			D^L(f_1) = d_t-r_1-l_1\\
			D^L(f_2)=d_t-r_2-l_2.
		\end{gather*} 
		For $i=1,2$ denote by $f_i'$ the bottom triangle of $1\cdot t$ such that $f_i'$ and $f_i$ are adjacent in $1\cdot t$ along the lower large edge of $f_i'$. 
		
		By Lemma \ref{lem:large_edges} the lower large edges of $f_i'$ are the equatorial edges of $1\cdot t$ which are of the same colour as the bottom diagonal of $1\cdot t$; see also Figure~\ref{fig:lower_tt}.		
		Therefore
		\[D^L(f_1+f_1') = D^L(f_2+f_2') = d_t-d_b-s_1-s_2\]
		where $(s_1,s_2) = \begin{cases}
			(r_1,r_2)& \text{if } d_b \text{ is blue}\\
			(l_1, l_2)& \text{if } d_b \text{ is red}.\end{cases}$.
	\end{proof}
	\begin{remark*}
		Lemma \ref{lemma:linear_dependence} does not hold for transverse taut triangulations which do not admit a veering structure.  One can check that if the lower large edges of the bottom faces of a tetrahedron $t$ are not the opposite equatorial edges of~$t$, then no nontrivial linear combination of the images of the faces of $t$ under~$D^L$ gives zero.
	\end{remark*}

	Let $Y$ be a spanning tree of the dual graph $X$ of $\mathcal{V} = ((T, F, E), \alpha, \nu)$. Let $F_Y$ be the subset of triangles which are dual to the edges of $X$ which are \emph{not} in $Y$ (non-tree edges). We define a $\ZH$-module homomorphism
	\[D_Y^L :\zz\lbrack H \rbrack^{F_Y} \rightarrow \zz \lbrack H\rbrack^E\]
	obtained from $D^L$ by deleting the columns corresponding to the edges of $Y$.
	\begin{proposition}\label{prop:non-tree_generate}
		Let $\mathcal{V} = (\mathcal{T}, \alpha, \nu)$ be a veering triangulation and $Y$ be a spanning tree of its dual graph $X$. The image of $D^L$ and that of $D^L_Y$ are equal.
	\end{proposition}
	\begin{proof}
		We say that a dual edge $h\cdot f \in H\cdot F$ is a \emph{linear combination of dual edges}  $h_{i_1}\cdot f_{i_1}, h_{i_2}\cdot f_{i_2}, \ldots, h_{i_k}\cdot f_{i_k} \in H\cdot F$, or \emph{in the span of these edges}, if $D^L(h\cdot f)$ is a linear combination with $\zz\lbrack H \rbrack$ coefficients of $D^L(h_{i_1}\cdot f_{i_1})$, $D^L(h_{i_2}\cdot f_{i_2}), \ldots, D^L(h_{i_k}\cdot f_{i_k})$. It is enough to prove that every tree edge is in the span of non-tree edges.
		
		By Lemma \ref{lemma:linear_dependence} each tree edge $h\cdot f \in H\cdot F$ is a linear combination of three dual edges that share a vertex with $h\cdot f$. In particular, the terminal edges of $Y$ --- there are at least two of them --- are in the span of non-tree edges. Now consider a subtree $Y'$  obtained from $Y$ by deleting its terminal edges. The terminal edges of $Y'$ can be expressed as linear combinations of non-tree edges and terminal edges of $Y$, hence as linear combinations of non-tree edges only. Since $Y$ is finite, we eventually exhaust all its edges.	 
	\end{proof}
	\begin{corollary}\label{cor:fast_comp}
		Let $Y$ be a spanning tree of the dual graph $X$ of a veering triangulation~$\mathcal{V}$. The taut polynomial $\Theta_\mathcal{V}$ is equal to the greatest common divisor of the maximal minors of the matrix $D_Y^L$.
	\end{corollary}
	\begin{proof}
		By Proposition \ref{prop:non-tree_generate} we obtain another presentation for the taut module
		\[ \zz\lbrack H\rbrack^{F_Y}\overset{D_Y^L}{\longrightarrow} \zz\lbrack H \rbrack^E \longrightarrow \module^L_\alpha(\mathcal{V}) \longrightarrow 0.\]
		Since  Fitting invariants of a finitely presented module do not depend on a chosen presentation \cite[pp. 58]{Northcott}, the greatest common divisor of the maximal minors of $D_Y^L$ is equal to the taut polynomial.
	\end{proof}
	\subsection{Computing the taut polynomial}\label{subsec:computation}
	In this subsection we present pseudocode for an algorithm which takes as an input a veering triangulation~$\mathcal{V}$ and outputs the taut polynomial $\Theta_{\mathcal{V}}$ of $\mathcal{V}$.
	
	In Section \ref{sec:example} we follow algorithm \texttt{TautPolynomial} applied to the veering triangulation \texttt{cPcbbbiht\_12} of the figure eight knot complement.
	\begin{algorithm}[H]
		\caption*{\textbf{Algorithm} \texttt{TautPolynomial}\\Computation of the taut polynomial of a veering triangulation}\label{alg:veering_poly}
		\begin{algorithmic}[1]
			\Statex 	\textbf{Input:} 
				A transverse taut veering triangulation $\mathcal{V} = ((T,F, E),\alpha,\nu)$ of a cusped 3-manifold $M$
			\Statex 	\textbf{Output:} The  taut polynomial $\Theta_\mathcal{V}$  
			\State $\Lambda: = $ \texttt{FacePairings}($\mathcal{V}, \lbrack \ \rbrack$) \Comment{Face Laurents encoding $\Vab$}
			\State $D: = $ the zero matrix with rows indexed by $E$ and columns by $F$ 
			\For{$e$ in $E$} 
			\State CurrentCoefficient := 1	
			\State $L:=$ list triangles on the left  of $e$, ordered  from bottom to top
			\State $R:=$ list triangles on the right of $e$, ordered from bottom to top	
			\For {$A$ in $\lbrace L,R \rbrace$} \Comment{Counting from 1, not 0}
			\State add CurrentCoefficient to the entry $(e, A\lbrack 1 \rbrack)$ of $D$
			\For{$i$ from 2 to length($A$)} 
			\State CurrentCoefficient := CurrentCoefficient$\cdot (\Lambda(A[i-1]))^{-1}$
			\State subtract CurrentCoefficient from the entry $(e,A\lbrack  i\rbrack)$ of $D$
			\EndFor
			\State CurrentCoefficient := 1
			\EndFor
			\EndFor
			\State $Y := \texttt{SpanningTree}(Y)$
			\State $D_Y:= D.$DeleteColumns$(Y)$ \Comment{Accelerate the computation}
			\State minors := $D_Y$.minors$(|E|)$
			\State \Return $\gcd(\mathrm{minors})$
		\end{algorithmic}
	\end{algorithm}
	\FloatBarrier
	\begin{proposition}\label{prop:veering_output}
		The output  of \emph{\texttt{TautPolynomial}} applied to a veering triangulation $\mathcal{V}$ is equal to the taut polynomial $\Theta_{\mathcal{V}}$ of $\mathcal{V}$. 
	\end{proposition}
	\begin{proof}
		The output of \texttt{FacePairings}($\mathcal{V}, \lbrack \ \rbrack$) encodes the triangulation $\mathcal{V}^{ab}$. The matrix~$D$  on line 15 of the algorithm {\texttt{TautPolynomial}}$(\mathcal{V}, \lbrack \ \rbrack )$ is equal to the presentation matrix $D^L$ of the taut module $\mathcal{E}_\alpha(\mathcal{V})$; see \eqref{pres:taut}. This follows from the following observations:
		\begin{itemize}
			\item  $1\cdot f \in H\cdot F$ has $h\cdot e$ in its boundary, for some $e\in E$, $h \in H$, if and only if $h^{-1}\cdot f$ is attached to $1\cdot e$ (this explains why we invert face Laurents in line 10)
			\item $1\cdot e$ is the lower large edge of $1\cdot f$ if and only if $1\cdot f$ is one of the two lowermost triangles adjacent to $1\cdot e$ (this explains why we add coefficients in line 8 and subtract in line 11),
			\item if $h\cdot f$ is attached to $1\cdot e$, then the product of $H$-pairings of triangles attached to $1\cdot e$ below $h\cdot f$, on the same side of $1\cdot e$, is equal to $h$ (this explains line~10).			
			Here by the $H$-pairing of a triangle of $\Vab$ we mean the $H$-pairing of its image under the covering map $\Vab \rightarrow \mathcal{V}$.
		\end{itemize}

		Deleting the tree columns of $D$, for some spanning tree $Y$ of the dual graph $X$ of~$\mathcal{V}$, gives another presentation matrix for the taut module $\mathcal{E}_\alpha(\mathcal{V})$ by Corollary~\ref{cor:fast_comp}. The greatest common divisor of its maximal minors is equal to the zeroth Fitting invariant of $\mathcal{E}_\alpha(\mathcal{V})$, that is the taut polynomial $\Theta_\mathcal{V}$ of~$\mathcal{V}$.	
	\end{proof}
	\section{The veering polynomials}\label{sec:veering_poly}
	Let $\mathcal{V}= ((T,F, E), \alpha, \nu)$ be a veering triangulation of a 3-manifold $M$. We still use the notation  
	\[H=\bigslant{H_1(M;\zz)}{\text{torsion}}.\]
	In Subsection \ref{subsec:graphs} we follow \cite[Section 4]{LMT} to recall  the definition of the flow graph associated to~$\mathcal{V}$. In Subsection \ref{subsec:veer_poly} we follow \cite[Section 3]{LMT} to recall the definition of the veering polynomial of $\mathcal{V}$. 
	
	The aim of this section is twofold. First, based on a computer search we show examples of veering triangulations which lack the lower-upper track symmetry. More precisely, 
	the authors of \cite{LMT} define the \emph{veering polynomial} and the \emph{flow graph} as invariants associated to the \emph{upper} train track of $\mathcal{V}$. In the previous section we have shown that the taut polynomial derived from the upper track and that derived from the lower track are equal (Proposition \ref{prop:lower=upper}). In this section we show  that this does \emph{not} hold for the veering polynomial (Proposition \ref{prop:not_equal}). Similarly, the flow graphs derived from the lower and upper train tracks of $\mathcal{V}$ are not necessarily isomorphic (Proposition~\ref{prop:nonisom_graphs}).
	
	Since the entries in the veering census  have a  coorientation fixed only up to a sign, we cannot assign a unique veering polynomial to a veering triangulation from the veering census. Instead, we get a pair of veering polynomials. Similarly, we get a pair flow graphs. The lower veering polynomial/flow graph of $(\mathcal{T}, \alpha, \nu)$ is the upper veering polynomial/flow graph of $(\mathcal{T}, -\alpha, \pm\nu)$, and vice versa.
	
	The second aim of this section is to present pseudocode for the computation of the lower veering polynomial  (Subsection \ref{subsec:comp_big}).  
	
	\subsection{Flow graphs}\label{subsec:graphs}
	Given a veering triangulation $\mathcal{V} = (\mathcal{T}, \alpha, \nu)$ the authors of~\cite{LMT} defined the \emph{flow graph} $\Phi^U$ of $\mathcal{V}$ \cite[Subsection 4.3]{LMT}.
	The vertices of $\Phi^{U}$ are in a bijective correspondence with edges $e \in E$. Corresponding to each $t \in T$ there are three edges of $\Phi^{U}$: 
	\begin{itemize}
		\item from the bottom diagonal $d_b$ of $t$ to the top diagonal $d_t$ of $t$,
		\item from the bottom diagonal $d_b$ of $t$ to the equatorial edges $e, e'$ of~$t$ which have a different colour than the top diagonal of $t$.
	\end{itemize}
	In this paper we call the obtained graph the \emph{upper} flow graph of~$\mathcal{V}$, hence the superscript $U$. We analogously define the \emph{lower flow graph} $\Phi^L$. Its vertices also correspond to the edges of the the veering triangulation.  Every $t \in T$ determines the following three edges of~$\Phi^{L}$:
	\begin{itemize}
		\item from the top diagonal $d_t$ of $t$ to the bottom diagonal $d_b$ of $t$,
		\item from the top diagonal $d_t$ of $t$ to the equatorial edges $e, e'$ of $t$ which have a different colour than the bottom diagonal of $t$.
	\end{itemize} 
	
	Using a computer search allowed us to find that
	\begin{proposition}\label{prop:nonisom_graphs}
		There exists a veering triangulation $\mathcal{V}$ whose lower and upper flow graphs are not isomorphic. 	 
	\end{proposition}
	\begin{proof} 
		The first entry of the veering census for which the upper and lower flow graphs are not isomorphic is the triangulation 
		\begin{center}$(\mathcal{T}, \pm \alpha, \pm \nu)$ = \texttt{hLMzMkbcdefggghhhqxqkc\_}\texttt{1221002}\end{center} of the manifold v2898. 
		
		The graphs are presented in Figure~\ref{fig:flow_graphs}. In \ref{fig:flow_graphs}(a)  there are two vertices of valency 6 (numbered with 4 and 6) which are joined to a vertex of valency~10 (numbered with 0),  while in \ref{fig:flow_graphs}(b) there is only one vertex of valency 6 (numbered with~6) which is joined to a vertex of valency~10 (numbered with 0). Hence the graphs are not isomorphic. 	\end{proof}
		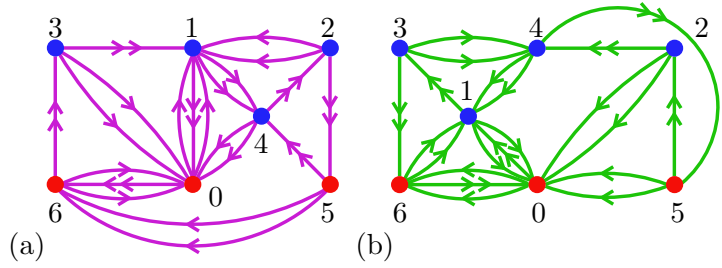
\begin{figure}[h]
			\begin{center}
				\includegraphics[width=0.615\textwidth]{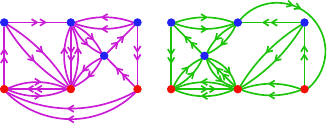} \vspace{1cm}  
				\put(-270,0){(a)}
				\put(-195,18){0}
				\put(-204,82){1}
				\put(-153,82){2}
				\put(-255,82){3}
				\put(-255,12){6}
				\put(-153,12){5}
				\put(-178,38){4}
				\put(-140,0){(b)}	
				\put(-74,12){0}
				\put(-75,82){4}
				\put(-13,82){2}
				\put(-126,82){3}
				\put(-126,12){6}
				\put(-22,12){5}
				\put(-101,58){1}
			\end{center}
			\caption{Flow graphs of \texttt{hLMzMkbcdefggghhhqxqkc\_1221002}. Double arrows join top diagonals to bottom diagonals of tetrahedra, or vice versa.}
			\label{fig:flow_graphs} 
		\end{figure}

	\subsection{Veering polynomials}\label{subsec:veer_poly}
	Let $\mathcal{V}= ((T,F, E), \alpha, \nu)$ be a veering triangulation. The matrix $D^L$ assigns to a tetrahedron a set of \emph{four} relations between its edges. By Lemma \ref{lemma:linear_dependence}, we can group triangles of a tetrahedron in pairs in such a way that $D^L$ evaluated on each pair equals
	\begin{equation}d_t-d_b-s_1-s_2, \label{eqn:big}\end{equation} where $d_t, d_b$ denote the top end the bottom diagonals of the tetrahedron, respectively, and $s_1, s_2$ --- its two equatorial edges of a different colour than~$d_b$. 
	
	Following \cite{LMT} we use this fact to define a $\zz\lbrack H\rbrack$-module homomorphism
	\begin{equation}\label{pres:veer}
		N^L:\zz\lbrack H \rbrack^T \rightarrow \zz\lbrack H \rbrack^E
		\end{equation}
	assigning to each tetrahedron of a veering triangulation $\Vab$ a linear combination of its edges of the form \eqref{eqn:big}. We call the cokernel of $N^L$ the \emph{lower veering module} and denote it by $\module^L_{\alpha, \nu}(\mathcal{V})$. The subscripts $\alpha, \nu$ reflect the fact that to define this module one needs both  the transverse taut structure $\alpha$  and the colouring $\nu$ on $\mathcal{V}$. 
	\begin{definition}
		The \emph{lower veering polynomial} $V^L_\mathcal{V}$ is the determinant of~$N^L$, that is the zeroth Fitting invariant of the lower veering module. 
	\end{definition}	
	Analogously the \emph{upper veering polynomial} $V^U_\mathcal{V}$ is the determinant of the map~$N^U$, which assigns to a tetrahedron with the top diagonal~$d_t$, the bottom diagonal $d_b$ and equatorial edges $w_1, w_2$ of a different colour than~$d_t$ a linear combination \[d_b-d_t-w_1-w_2.\]
	
	\begin{remark}\label{remark:not_up_to_a_unit}
		In \cite[Section 3]{LMT} the (upper) veering polynomial is well-defined as an element of $\ZH$, and not just up to a unit. This is accomplished by identifying a tetrahedron of $\Vab$ with its bottom diagonal. Then the map $N^U$ has $\ZH^E$ as both the domain and codomain. The upper veering polynomial $V^U_\mathcal{V}$ is then equal to the determinant of $N^U$, where the basis for  the domain and codomain is chosen to be the same.
		
		For the lower veering polynomial we identify a tetrahedron of $\Vab$ with its top diagonal.
		By our conventions for labelling ideal simplices of $\Vab$ (explained in Subsection~\ref{subsec:labelling}) under this identification the basis $1\cdot T$ for $\ZH^T$ and the basis $1\cdot E$ for $\ZH^E$ differ at most by a permutation.
	\end{remark}	
	Both veering modules have~$n$ generators and $n$ relations, hence their zeroth Fitting ideals are principal. This has obvious computational advantages. Moreover, the authors of~\cite{LMT} proved that the upper veering polynomial can be interpreted as the Perron polynomial of $\Phi^{U}$ (where the edges of $\Phi^{U}$ are labelled with certain elements of~$H$) \cite[Theorem~4.8]{LMT}. This allowed them to employ the results of McMullen on the Perron polynomials of graphs \cite[Section 3]{McMullen_Perron} to compute the growth rate of the closed orbits of the pseudo-Anosov flow associated to $\mathcal{V}$ in \cite{LMT_flow}. 
	\subsection{Computing the veering polynomials}\label{subsec:comp_big}
	In this subsection we present pseudocode for an algorithm which takes as an input a veering triangulation and returns its lower veering polynomial. 
	
	In Section \ref{sec:example} we follow this algorithm applied to the veering triangulation \texttt{cPcbbbiht\_12} of the figure eight knot complement.
	
	\begin{algorithm}[H]
		\caption*{\textbf{Algorithm} \texttt{LowerVeeringPolynomial}\\Computation of the lower veering polynomial}\label{alg:big_veering_poly}
		\begin{algorithmic}[1]
			\Statex 	\textbf{Input:} A transverse taut veering triangulation $\mathcal{V} = ((T,F, E),\alpha,\nu)$ of a cusped 3-manifold $M$
			\Statex 	\textbf{Output:} The lower veering polynomial $V^L_\mathcal{V}$ of $\mathcal{V}$
			\State permute elements of $T$ so that $E\lbrack i \rbrack$ is the top diagonal  of $T\lbrack i \rbrack$ 
			\State $\Lambda: = $ \texttt{FacePairings}$(\mathcal{V}, \lbrack \ \rbrack)$ \Comment{Face Laurents encoding $\Vab$}
			\State $N: = $ the zero matrix with rows indexed by $E$ and columns by $T$ 
			\For{$e$ in $E$}
			\State CurrentCoefficient := 1
			\State $L:=$ triangles on the left  of $e$, ordered  from bottom to top 
			\State $R:=$ triangles on the right of $e$, ordered from bottom to top
			\State $BT: = $ tetrahedron immediately below $L[1]$
			\State add CurrentCoefficient to the entry $(e, BT)$ of $N$
			\State $TT:= $ tetrahedron immediately above $L[\text{length}(L)]$
			\State TopCoefficient:=$\prod\limits_{i=1}^{\text{length}(L)} (\Lambda(L[i]))^{-1}$
			\State subtract TopCoefficient from the entry $(e,TT)$ of $N$	
			\For{$A$ in $\lbrace L,R\rbrace$}
			\For{$i$ from 1 to length($A$)$-1$} \Comment{Counting from 1, not 0}
			\State $T:=$ tetrahedron immediately above $A[i]$	
			\State CurrentCoefficient := CurrentCoefficient $\cdot (\Lambda(A[i]))^{-1}$
			\If {$i>1$} \label{code:skipping-first}
			\State subtract CurrentCoefficient from the entry
			$(e,T)$ of $N$
			\EndIf
			\EndFor
			\State CurrentCoefficient :=1
			\EndFor
			\EndFor
			\State \Return determinant of $N$
		\end{algorithmic}
	\end{algorithm}
	\begin{proposition} \label{prop:big_algorithm_correct}The output of \emph{\texttt{LowerVeeringPolynomial}} applied to a veering triangulation $\mathcal{V}$ is equal to the lower veering polynomial of $\mathcal{V}$.
	\end{proposition}
	\begin{proof}
		We claim that the matrix $N$ on line 23 of the algorithm is equal to $N^L$ given in \eqref{pres:veer}. Hence its determinant is equal to the lower  veering polynomial of~$\mathcal{V}$. The proof is similar to that of Proposition \ref{prop:veering_output}. The main difference here is that when we explore an edge $e \in E $ we do not take into account all tetrahedra attached to it, but only the one immediately below $e$, immediately above~$e$ and the  ones on the sides of $e$ which have a bottom diagonal of a different colour than~$e$. However, by Corollary~\ref{cor:triangles_of_same_colour} we know that the only~side tetrahedra of~$e$ which do not have a bottom diagonal of a different colour than $e$ are the two lowermost side tetrahedra of $e$. This explains line~\ref{code:skipping-first} of \texttt{LowerVeeringPolynomial}.
		
		Line 1 of the \texttt{LowerVeeringPolynomial} ensures that the polynomial is correctly computed not only up to a sign; see Remark \ref{remark:not_up_to_a_unit}.
	\end{proof}
	An analogous algorithm can be written for the upper veering polynomial. Alternatively,   by Remark \ref{remark:changes_on_tracks} to compute the upper veering polynomial of $(\mathcal{T}, \alpha, \nu)$ we can apply \texttt{LowerVeeringPolynomial}  to the triangulation $(\mathcal{T}, -\alpha, -\nu)$. 
	\begin{remark*}Let $\mathcal{V} = (\mathcal{T}, \alpha, \nu)$ and $-\mathcal{V} = (\mathcal{T}, -\alpha, -\nu)$.	
		The outputs of \texttt{Lower\linebreak VeeringPolynomial}$(\mathcal{V})$ and \texttt{LowerVeeringPolynomial}$(-\mathcal{V})$ are equal to~$V^L_\mathcal{V}$ and $V^U_{\mathcal{V}}$, respectively, but written with respect to \emph{different} bases of $H$. The change of basis is given by $u_i \mapsto u_i^{-1}$. This follows from Definition \ref{defn:H-pairings}. Moreover, \[V^L_{-\mathcal{V}}(u_1^{-1}, \ldots, u_r^{-1}) =  V^U_\mathcal{V}(u_1, \ldots, u_r)\] only up to a unit; compare with Remark \ref{remark:not_up_to_a_unit}.
	\end{remark*}
	
	Using an implementation of \texttt{LowerVeeringPolynomial} we have found that
	\begin{proposition}\label{prop:not_equal}
		There exists a veering triangulation $\mathcal{V}$ for which $V^L_\mathcal{V}$ and $V^U_\mathcal{V}$ are not equal up to a unit in $\ZH$.
	\end{proposition}
	\begin{proof}
		The first entry $(\mathcal{T}, \pm \alpha, \pm \nu)$ of the veering census for which (after fixing the signs for $\alpha, \nu$) we have
		\[V^L_\mathcal{V} \neq V^U_\mathcal{V}\]
		in $\bigslant{\ZH}{H}$ is $(\mathcal{T}, \pm \alpha, \pm \nu)$ =\texttt{iLLLAQccdffgfhhhqgdatgqdm\_21012210}. This is a veering triangulation of the 3-manifold t10133. Its lower and upper veering polynomials are up to a unit equal to
		\begin{align*}(1-u)^{-1}&(1-u^{25})(1-u^{13})\\&(1-u+u^2-u^3+u^4-u^5+u^6)(1-u^2-u^7-u^{12}+u^{14})\end{align*}
		and
		\begin{align*}(1-u)^{-1}&(1-u^{29})(1-u^9)\\&(1-u+u^2-u^3+u^4-u^5+u^6)(1-u^2-u^7-u^{12}+u^{14}).\end{align*}
		
		Their greatest common divisor \[(1-u+u^2-u^3+u^4-u^5+u^6)(1-u^2-u^7-u^{12}+u^{14})\]is equal to the taut polynomial $\Theta_{\mathcal{V}}$ of $\mathcal{V}$.
	\end{proof}
	\begin{remark}The flow graphs of the triangulation from the proof of Proposition \ref{prop:not_equal} are not isomorphic. In fact, one of them is planar, and the other is not. 
		\end{remark}
	\begin{remark} Recall the veering triangulation 
		\begin{center}$(\mathcal{T}, \pm \alpha, \pm \nu)$ = \texttt{hLMzMkbcdefggghhhqxqkc\_}\texttt{1221002}\end{center} of the manifold v2898. 
		In the proof of Proposition \ref{prop:nonisom_graphs} we showed that its upper and lower flow graphs are not isomorphic. 		
		However, its lower and upper veering polynomial are both (up to a unit) equal to \[(1+u)(1-20u+u^2).\] 
		\end{remark}
	
	There are even veering triangulations for which one veering polynomial vanishes and the other does not.
	\begin{example}\label{ex:veering_vanishes}
		We consider the triangulation \[(\mathcal{T}, \pm \alpha, \pm \nu) = \texttt{lLLLAPAMcbcfeggihijkktshhxfpikaqj\_20102220020}.\]
		
		With some choice of signs on $\alpha, \nu$ we have
		\[V^L_\mathcal{V} = (u - 1)^2  (u + 1)^3  (u^2 - u + 1)  (u^4 + 1)\]
		up to a unit and
		\[V^U_\mathcal{V} = 0.\]
	\end{example}
	\begin{remark}\label{remark:when_vanishes}
		By the results of Landry, Minsky and Taylor the taut polynomial divides the upper veering polynomial \cite[Theorem 6.1 and Remark~6.18]{LMT} and hence also the lower veering polynomial. The remaining factors of the lower/upper veering polynomial are related to a special family of \mbox{1-cycles} in the dual graph of the veering triangulation, called the lower/upper \emph{AB-cycles} \cite[Section 4]{LMT}. We refer the reader to \cite[Subsection~6.1]{LMT} to find out the formula for the extra factors.
		
		If for a veering triangulation $\mathcal{V}$ we have that $\Theta_\mathcal{V} \neq 0$ and $V^{L/U}_\mathcal{V} = 0$, then~$\mathcal{V}$ has a lower/upper AB-cycle of even length whose homology class in $H$ is trivial.	
		From Proposition \ref{prop:not_equal} it follows that  the homology classes of the lower and upper AB-cycles are not always paired so that one is the inverse of the other.
	\end{remark}

	\section{Example: the veering triangulation of the figure eight knot complement} \label{sec:example}
	In this section we follow algorithms \texttt{TautPolynomial} and \texttt{LowerVeering} \texttt{Polynomial} on the veering triangulation \texttt{cPcbbbiht\_12} of the figure eight knot complement. 
	\subsection{Triangulation of the maximal free abelian cover}
	Let $\mathcal{V}$ denote the veering triangulation \texttt{cPcbbbiht\_12} of the figure eight knot complement. First we  follow the algorithm \texttt{FacePairings} in order to encode the triangulation $\Vab$ of the maximal free abelian cover of the figure eight knot complement.
	
	We find the branch equations matrix $B$ for $\mathcal{V}$. Figure \ref{fig:m004_edges} shows triangles attached to the edges $e_0, e_1$ of $\mathcal{V}$. Using this we see that
	\[B = \begin{blockarray}{l c c}
		&  e_0   &   e_1      \\
		\begin{block}{l [c c]}
			f_0 & 0 & 1 \\
			f_1 & 1 & 0 \\
			f_2 & 0 & -1 \\
			f_3 & -1 & 0 \\
		\end{block}
	\end{blockarray}.\]
	\begin{figure}[h]
		\begin{center}
			\includegraphics[width=0.525\textwidth]{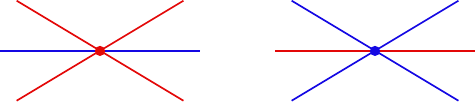}
			\put(-178,-15){\fbox{$e_0$}}
			\put(-53,-15){\fbox{$e_1$}}
			\put(-174,5){$t_1$}
			\put(-174,35){$t_0$}
			\put(-205,12){$t_0$}
			\put(-205,29){$t_1$}
			\put(-143,12){$t_0$}
			\put(-143,29){$t_1$}
			\put(-48,5){$t_0$}
			\put(-48,35){$t_1$}
			\put(-79,12){$t_1$}
			\put(-79,29){$t_0$}
			\put(-17,12){$t_1$}
			\put(-17,29){$t_0$}
			\put(-220,-5){$f_0$}
			\put(-228,20){$f_1$}
			\put(-220,45){$f_2$}
			\put(-131,-5){$f_2$}
			\put(-123,20){$f_3$}
			\put(-131,45){$f_0$}
			\put(-94,-5){$f_1$}
			\put(-102,20){$f_0$}
			\put(-94,45){$f_3$}
			\put(-6,-5){$f_3$}
			\put(2,20){$f_2$}
			\put(-6,45){$f_1$}
		\end{center}
		\caption{Cross-sections of the neighbourhoods of the edges $e_0$, $e_1$ of $\mathcal{V}$. The colours of edges and triangles are indicated.}
		\label{fig:m004_edges} 
	\end{figure}
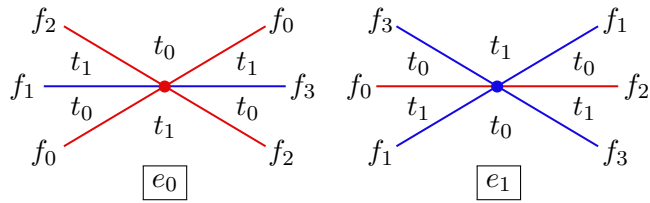
	
	Using Figure \ref{fig:m004_edges} we can draw the dual graph $X$ of $\mathcal{V}$. It is presented in Figure~\ref{fig:m004_dual_graph}. As a spanning tree $Y$ of~$X$ we choose $\lbrace f_0 \rbrace$.
	\begin{figure}[h]
		\begin{center}
			\includegraphics[width=0.133\textwidth]{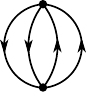}
			\put(-30,-10){$t_0$}
			\put(-30,62){$t_1$}
			\put(-67,28){$f_0$}
			\put(-49,28){$f_2$}
			\put(-15,28){$f_1$}
			\put(2,28){$f_3$}
		\end{center}
		\caption{The dual graph of \texttt{cPcbbbiht\_12}.}
		\label{fig:m004_dual_graph} 
	\end{figure}
	The matrix $B_Y$ is obtained from $B$  by deleting its first row, corresponding to $f_0$. Let $S$ be the Smith normal form $S$ of $B_Y$. It satisfies $S = UB_Y V$, where
	\[U = \begin{blockarray}{l c c c }
		&  f_1   &   f_2 & f_3 \\
		\begin{block}{l [c c c]}
			&-1 & 0 & 0\\
			&0 & 1 & 0 \\
			&1 & 0 & 1\\
		\end{block}
	\end{blockarray}.\]
	Since $S$ is of rank 2, the face Laurents of the non-tree edges are determined by the last row of $U$. All face Laurents for $\mathcal{V}$  relative to the fundamental domain determined by $Y = \lbrace f_0 \rbrace$ are listed in Table \ref{tab:faceLaur}.
	
	\begin{table}[h]
		\begin{tabular}{|c|c|c|c|c|}
			\hline
			face&	$f_0$ & $f_1$ & $f_2$ & $f_3$\\\hline
			face Laurent&	1& $u$ & 1 & $u$\\\hline
		\end{tabular}
		\caption{The face Laurents encoding $\Vab$.}
		\label{tab:faceLaur}
	\end{table}
	
	Using Table \ref{tab:faceLaur} and Figure \ref{fig:m004_edges} we draw the triangles and the tetrahedra attached to the edges $1\cdot e_0$ and $1\cdot e_1$ of $\Vab$ in Figure \ref{fig:m004_edges_upstairs}.
	\begin{figure}[h]
		\begin{center}
			\includegraphics[width=0.657\textwidth]{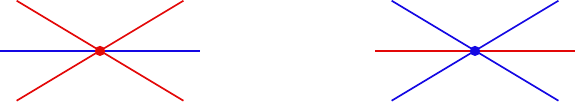}
			\put(-235,-15){\fbox{$1\cdot e_0$}}
			\put(-59,-15){\fbox{$1\cdot e_1$}}
			\put(-232,5){$1\cdot t_1$}
			\put(-233,35){$u\cdot t_0$}
			\put(-268,12){$1\cdot t_0$}
			\put(-268,29){$u\cdot t_1$}
			\put(-199,12){$1\cdot t_0$}
			\put(-199,29){$u\cdot t_1$}
			\put(-56,5){$1\cdot t_0$}
			\put(-61,35){$u^2\cdot t_1$}
			\put(-95,12){$u\cdot t_1$}
			\put(-95,29){$u\cdot t_0$}
			\put(-22,12){$u\cdot t_1$}
			\put(-22,29){$u\cdot t_0$}
			\put(-283,-10){$1\cdot f_0$}
			\put(-295,20){$1\cdot f_1$}
			\put(-288,50){$u\cdot f_2$}
			\put(-186,-10){$1\cdot f_2$}
			\put(-173,20){$1\cdot f_3$}
			\put(-186,50){$u\cdot f_0$}%
			\put(-108,-10){$1\cdot f_1$}
			\put(-121,20){$u\cdot f_0$}
			\put(-111,50){$u\cdot f_3$}
			\put(-8,-10){$1\cdot f_3$}
			\put(3,20){$u\cdot f_2$}
			\put(-8,50){$u\cdot f_1$}
		\end{center}
		\caption{Cross-sections of the neighbourhoods of the edges $1\cdot e_0$, $1\cdot e_1$ of~$\mathcal{V}^{ab}$. The colours of edges and triangles are indicated.}
		\label{fig:m004_edges_upstairs} 
	\end{figure}
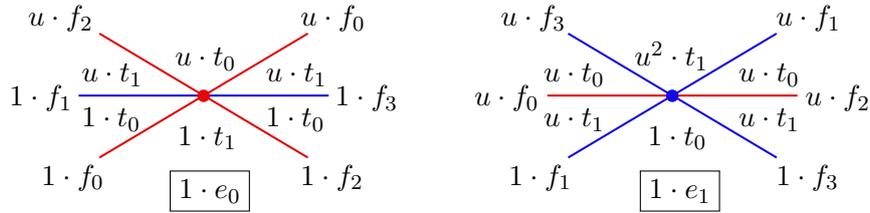
	\subsection{The taut polynomial}
	To find the presentation matrix $D^L$ it is enough to know the (inverses of) Laurent coefficients of the triangles attached to $1\cdot e_0$ and $1\cdot e_1$. They can be read off from Figure \ref{fig:m004_edges_upstairs}. Note that~$1\cdot e_i$ is lower large only in its two lowermost triangles. Recall that by Corollary \ref{cor:fast_comp} the taut polynomial of $\mathcal{V}$ is equal to the greatest common divisor of the matrix $D^L_Y$, obtained from $D^L$ by deleting its first column, corresponding to the tree $Y = \lbrace f_0 \rbrace$.
	We have
	\[D^L_Y(u) = \begin{blockarray}{l c c c }
		&  f_1   &   f_2 & f_3 \\
		\begin{block}{l [c c c]}
			e_0&-1 & 1-u^{-1} & -1\\
			e_1&1-u^{-1} & -u^{-1} & 1-u^{-1} \\
		\end{block}
	\end{blockarray}\]
	and hence 
	\[\Theta_\mathcal{V} = 1-3u+u^2\]
	up to a unit in $\zz\lbrack u^{\pm1} \rbrack$.
	\subsection{The veering polynomial}
	To find the presentation matrix $N^L$ it is enough to know the (inverses of) Laurent coefficients of the tetrahedra attached to $1\cdot e_i$. 
	They can be read off from Figure \ref{fig:m004_edges_upstairs}. Recall that among side tetrahedra we only take into account the ones which have the bottom diagonal of a different colour than $1\cdot e_i$. By Corollary~\ref{cor:triangles_of_same_colour} this boils down to skipping the lowermost tetrahedra. We get
	\[N^L(u) = \begin{blockarray}{l c c }
		&  t_0  &   t_1  \\
		\begin{block}{l [c c]}
			e_0&-u^{-1} & 1-2u^{-1} \\
			e_1&1-2u^{-1} & -u^{-2}  \\
		\end{block}
	\end{blockarray}.\]
	Thus 
	\[V_\mathcal{V}^L = -(u^{-3} - 4u^{-2}+4u^{-1}-1)\]
	The minus sign in front is necessary because $e_0$ is the top diagonal of $t_1$ and~$e_1$ is the top diagonal of $t_0$; see Remark \ref{remark:not_up_to_a_unit}.
	Up to a unit we have
	\[V_\mathcal{V}^L = (u-1)\cdot \Theta_{\mathcal{V}}.\]
	\section{The Teichm\"uller polynomial}\label{sec:Teich}
	Let $N$ be a finite volume, oriented, hyperbolic 3-manifold. There is norm on $H^1(N;\rr) \cong H_2(N,\partial N;\rr)$, called the \emph{Thurston norm}, whose unit ball is a polytope with rational vertices \cite[Theorem 2]{Thur_norm}. It may admit some (top-dimensional) faces, called \emph{fibred faces}, which encode the ways in which $N$ fibres over the circle.
		
	Let $\face{F}$ be a fibred face of the Thurston norm ball in $H^1(N;\rr)$. Any integral primitive class in the interior of the cone $\rr_+\hspace{-0.1cm}\cdot \face{F}$ determines a fibration
	\[S \rightarrow N \rightarrow S^1\]
	over the circle \cite[Theorem 3]{Thur_norm}. We can express $N$ as the mapping torus
	\[N =\bigslant{\left(S \times \lbrack 0, 1 \rbrack\right)}{\lbrace (x,1)\sim (\psi(x),0) \rbrace},\]
	where $\psi$ is a pseudo-Anosov homeomorphism of the surface $S$ \cite[Proposition 2.6]{Thurston_map_tor}. It is called the \emph{monodromy of a fibration}.
	The monodromy determines the \emph{suspension flow} on $N$ defined as the unit speed flow along the curves $\lbrace x \rbrace \times \lbrack 0 ,1 \rbrack$. It admits a finite number of closed \emph{singular orbits} $\ell_1, \ldots, \ell_k \subset N$. The singular orbits arise from the prong-singularities of the invariants foliations of $\psi$ in $S$.	
	Fried showed that (up to isotopy) the flow does not depend on the chosen integral homology class in  $\rr_+\hspace{-0.1cm}\cdot \face{F}$ \cite[Theorem~14.11]{Fried_suspension}. Therefore the set of singular orbits depends only on the face. We set
	\[\mathrm{sing}(\face{F}) = \lbrace \ell_1, \ldots, \ell_k\rbrace\]
	\begin{definition}
		Let $\face{F}$ be a fibred face of the Thurston norm ball in $H^1(N;\rr)$. We say that $\face{F}$ is \emph{fully-punctured} if the set $\mathrm{sing}(\face{F})$ is empty.
	\end{definition}
	McMullen defined a polynomial invariant $\Theta_{\texttt{F}}$ of $\face{F}$, called the \emph{Teichm\"uller polynomial}. Landry, Minsky and Taylor proved that it is closely related to the taut polynomial of the veering triangulation associated to $\face{F}$ \cite[Proposition 7.2]{LMT}. 	An algorithm to compute the taut polynomial is given in Subsection \ref{subsec:computation}. In this section we use it to give an algorithm to compute $\Theta_{\texttt{F}}$.
	
\subsection{Veering triangulations associated to fibred faces}\label{subsec:veer:fibred}
Let $N$ be a finite volume, oriented, hyperbolic 3-manifold. Let $\face{F}$ be a fibred face of the Thurston norm ball in $H^1(N;\rr)$. If we pick a fibration lying over $\face{F}$ we can follow Agol's algorithm \cite[Section~4]{Agol_veer} to construct a layered veering triangulation $\mathcal{V}$ of $M =  N \bez \mathrm{sing}(\face{F})$.
The fact that~$\mathcal{V}$ does not depend on the chosen fibration from $\rr_+\hspace{-0.1cm}\cdot \face{F}$ is proven in \cite[Proposition~2.7]{MinskyTaylor}. In this section we change the previous notation and set
\begin{gather*}H_N := \bigslant{H_1(N;\zz)}{\text{torsion}}\\
	H_M := \bigslant{H_1(M;\zz)}{\text{torsion}}.
\end{gather*} The inclusion of $M$ into $N$ induces an epimorphism $i_\ast: H_M \rightarrow H_N$.  
\begin{lemma}[Proposition 7.2 in \cite{LMT}] \label{lem:computing_Teich}
	With the notation as above we have
	\[\pushQED{\qed} 
	\Theta_{\mathtt{F}} = i_\ast(\Theta_{\mathcal{V}}). \qedhere
	\popQED\] 	
\end{lemma}	
In particular, if the face \face{F} is fully-punctured, then its Teichm\"uller polynomial is equal to the taut polynomial of $\mathcal{V}$.
We gave an algorithm to compute the taut polynomial in Subsection \ref{subsec:computation}. In Subsections \ref{subsec:epimorphism} and \ref{subsec:boundary_cycles} we explain how to compute $i_\ast(\Theta_\mathcal{V})$. 
	\subsection{Classical (fully-punctured) examples}
	The majority of the computations of the Teichm\"uller polynomial previously known in the literature concern only fully-punctured fibred faces. Table \ref{tab:classical_examples} presents the output of the algorithm \texttt{TautPolynomial} applied to some of these examples.
	
	\begin{table}[h]
		\begin{tabular}{|c | c|}
			\hline
			\multicolumn{2}{|c|}{Example 1}\\ \hline
			\rule{0pt}{0.4cm}Source of the example& McMullen \cite[Subsection 11.I]{McMullen_Teich}\\ \hline
			\rule{0pt}{0.4cm}	Polynomial in the source& $1-u-ut-ut^{-1}+u^2$\\ \hline
			\rule{0pt}{0.4cm}Veering triangulation $\mathcal{V}$& \texttt{eLMkbcddddedde\_2100}\\ \hline
			\rule{0pt}{0.4cm} \texttt{TautPolynomial}($\mathcal{V}$)& $a^2b-a^2-ab-b^2+b$\\ \hline 
			\rule{0pt}{0.4cm} Change of basis& $t \mapsto ab^{-1}$, $u \mapsto a$\\ \hline \hline
			\multicolumn{2}{|c|}{Example 2}\\ \hline
			\rule{0pt}{0.4cm}Source of the example& McMullen \cite[Subsection 11.II]{McMullen_Teich}\\ \hline
			\rule{0pt}{0.4cm}	Polynomial in the source& $t^{-2}-ut-u-ut^{-1}-ut^{-2}-ut^{-3}+u^2$\\ \hline
			\rule{0pt}{0.4cm}Veering triangulation $\mathcal{V}$& \texttt{ivvPQQcfghghfhgfaddddaaaa\_20000222}\\ \hline
			\rule{0pt}{0.4cm} \texttt{TautPolynomial}($\mathcal{V}$)& $ab^4-a^2b^2+ab^3+ab^2+ab-b^2+a$\\ \hline \rule{0pt}{0.4cm} Change of basis& $t \mapsto b$, $u \mapsto ab^{-1}$\\ \hline \hline
			\multicolumn{2}{|c|}{Example 3}\\ \hline
			\rule{0pt}{0.4cm}Source of the example& Lanneau \& Valdez \cite[Subsection 7.2]{LanVal}\\ \hline
			\rule{0pt}{0.4cm}	Polynomial in the source& $u^2 - ut_At_B-ut_B-u-ut_A^{-1}+t_B$\\ \hline
			\rule{0pt}{0.4cm}
			Veering triangulation $\mathcal{V}$& \texttt{gvLQQcdeffeffffaafa\_201102} \\ \hline
			\rule{0pt}{0.4cm} \texttt{TautPolynomial}($\mathcal{V}$)& $a^2bc^2-abc-ac^2-ab-ac+1$\\ \hline
			\rule{0pt}{0.4cm} Change of basis& $t_A \mapsto bc^{-1}$, $t_B \mapsto b^{-1}c^2$, $u \mapsto ac^2$\\ \hline
			
		\end{tabular}
		\caption{Using \texttt{TautPolynomial} to calculate some  Teichm\"uller polynomials previously known in the literature.}
		\label{tab:classical_examples}
	\end{table}
	\FloatBarrier
	\subsection{Epimorphism $i_\ast: H_M \rightarrow H_N$} 
	\label{subsec:epimorphism}
	We follow the notation introduced in Subsection~\ref{subsec:veer:fibred}. By Lemma \ref{lem:computing_Teich} in order to compute the Teichm\"uller polynomial of $\face{F}$ we need  to establish a way of computing $i_\ast: H_M \rightarrow H_N$.
	
	A strategy to do that is as follows. Fix a monodromy $\psi: S\rightarrow S$ of a fibration of $N$ over
	the circle such that the homology class $\lbrack S \rbrack \in H_2(N, \partial N;\rr)$ lies in $\rr_+ \hspace{-0.1cm}\cdot \face{F}$. Recall that $M = N \bez \mathrm{sing}(\face{F})$. Define $\mathring{S} = S \cap M$ and let $\mathring{\psi}$ be the restriction of $\psi$ to $\mathring{S}$. 
		
	In the compact model $M$ has at least $k$ toroidal boundary components. There are $k$ of them, $T_1, \ldots, T_k$, whose inclusions into $N$ are the boundaries of the regular neighbourhoods of the singular orbits $\ell_1, \ldots, \ell_k$ of $\face{F}$. For \mbox{$j \in \lbrace 1, \ldots, k \rbrace$} let $\gamma_j$ denote the intersection $\mathring{S} \cap T_j$. The curve $\gamma_j$ might not be connected, but its connected components are parallel in $T_j$. These curves determine Dehn filling coefficients on $T_j$'s which produce the manifold $N$ out of $M$ and restore the surface $S$ from $\mathring{S}$. Therefore a presentation for $H_1(N;\zz)$ can be obtained from a presentation for $H_1(M;\zz)$ by adding extra relations which say that the homology classes of curves $\gamma_1, \ldots, \gamma_k$ are trivial.
	
	\begin{remark*}
		If $\gamma_j$ is not connected, the relation $\gamma_j = 1$ is a multiple. However, since we eventually use only the torsion-free part $H_N$ of $H_1(N;\zz)$ killing $\gamma_j$ is sufficient.
	\end{remark*}

	Recall from Subsection~\ref{subsec:encoding} that our presentation for $H_M$  uses the 2-complex $\mathcal{D}$ dual to $\mathcal{V}$. To find an explicit expression of $H_N$ as a quotient of $H_M$ it is enough to find a collection $C = \lbrace c_1, \ldots, c_k \rbrace$ of simplicial 1-cycles in $\mathcal{D}$ which are homologous to $\lbrace \gamma_1, \ldots, \gamma_k \rbrace$, respectively. Then
	\[H_N = H_M^C.\]
	The meaning of the superscript $C$ here is as in Subsection \ref{subsec:encoding}. The process of finding~$C$ is explained in Subsection \ref{subsec:boundary_cycles}. Below we assume that $C$ is already known.
	
	Let $n$ be the number of tetrahedra in the veering triangulation $\mathcal{V}$ of $M$. Let $X$ be the dual graph of $\mathcal{V}$. Recall that at the very beginning the algorithm \texttt{FacePairings} fixes a spanning tree $Y$ of $X$. Denote by $X_Y$ the  graph obtained from $X$ by contracting $Y$ to a point.	
	The output of \texttt{FacePairings}($\mathcal{V}, \lbrack \ \rbrack$, return type = ``matrix'') is a pair $(U,r)$, where $r$ equals the rank of $H_M$ and the last $r$ columns of the inverse $U^{-1} \in \mathrm{GL}(n+1,\zz)$ give the expressions for the basis elements of $H_M$ as simplicial 1-cycles in  $X_Y$. 
	
	Let $C$ be a collection of 1-cycles in $X$ homologous to $\lbrace \gamma_1, \ldots, \gamma_k\rbrace$. The output of \texttt{FacePairings}($\mathcal{V}$, $C$, return type = ``matrix'') is a pair $(U',s)$, where $s$ equals the rank of $H_M^C$ and the last $s$ rows of the matrix $U' \in \mathrm{GL}(n+1,\zz)$ encode $H_M^C$-pairings for the edges of $X_Y$. 
	
	Let $A_1$ be the matrix obtained from $U^{-1}$ by deleting its first $n+1-r$ columns. Let $A_2$ be the matrix obtained from $U'$ by deleting its first $n+1-s$ rows. Then the matrix $A = A_2\cdot A_1$ represents the epimorphism $i_\ast: H_M \rightarrow H_M^C$ written with respect to the bases of $H_M, H_M^C$ fixed by the algorithm \texttt{FacePairings}.

	\subsection{Boundary components  as dual cycles}\label{subsec:boundary_cycles}
	 The goal of this subsection is to present an algorithm which given a veering triangulation $\mathcal{V}$ and a surface $\mathring{S}$ carried by $\mathcal{V}$ outputs a collection of simplicial 1-cycles in the 2-complex $\mathcal{D}$ dual to $\mathcal{V}$ which are homologous to the boundary components of $\mathring{S}$.

	We follow the notation introduced in Subsections \ref{subsec:veer:fibred} and \ref{subsec:epimorphism}.		
	Since the veering triangulation $\mathcal{V}$ can be built by layering ideal tetrahedra on an ideal triangulation of the given fibre $\mathring{S}$ \cite[Section 4]{Agol_veer}, there exists a nonzero, nonnegative integral solution $w = (w_1, \ldots, w_{2n})$ to the system of branch equations of $\mathcal{V}$ such that the surface  \[S^w =\sum\limits_{i=1}^{2n}w_if_i\]
	represents the relative homology class of $\lbrack \mathring{S}\rbrack $  in $H_2(M, \partial M;\rr)$. Note that $w$ is \emph{not} unique. 
	
	 Recall that we consider $M$ in the compact model. It has $b\geq k$ boundary components $T_1, \ldots, T_b$ and we order them so that the first $k$ come from the singular orbits $\lbrace \ell_1, \ldots, \ell_k \rbrace$ in $N$ and the last $b-k$ come from $\partial N$. The boundary components $\gamma_j$ of the surface $S^w$ are carried by the boundary track; see Subsection \ref{subsec:boundary_track}. The tuple $w$ endows each boundary track $\beta_1, \ldots, \beta_k$, $\beta_{k+1}, \ldots, \beta_b$ with a nonnegative integral transverse measure which encodes the boundary components $\gamma_j = \partial S^w \cap T_j$ for $1\leq j \leq b$.	
	The general idea to find the cycle $c_j$ homologous to $\gamma_j$ is as follows.
	\begin{enumerate} \item Perturb~$\gamma_j$ slightly, so that it becomes transverse to the boundary track. 
		\item Push the (perturbed) $\gamma_j$ away from the boundary of $M$ into the dual graph $X$.  
		\end{enumerate}
	First we define an auxiliary object, the \emph{dual boundary graph} $X^\beta$.
	\begin{definition}
		Let $(\mathcal{T}, \alpha)$ be a (truncated) transverse taut ideal triangulation of a (compact) 3-manifold~$M$. The \emph{dual boundary graph} $X^\beta$ is the graph contained in $\partial M$ which is dual to the boundary track $\beta$ of $(\mathcal{T}, \alpha)$. It is oriented by  $\alpha$.
	\end{definition}
	If $b \geq 1$ then the dual boundary graph is disconnected, with connected components $X^\beta_1, \ldots, X^\beta_b$ such that $X^\beta_j$ is dual to the boundary track $\beta_j$.
	If an edge of $X^\beta$ is dual to a branch of $\beta$ lying in $f \in F$, then we label it with~$f$. Hence for every $f \in F$ there are three edges of $X^\beta$ labelled with~$f$.
	
	\begin{example}The dual boundary graph for the veering triangulation \linebreak \texttt{cPcbbbiht\_12} of the figure eight knot complement is given in Figure \ref{fig:boundary_torus_m004}.
	\end{example}
	
	\begin{figure}[h]
		\begin{center}
			\includegraphics[width=0.9\textwidth]{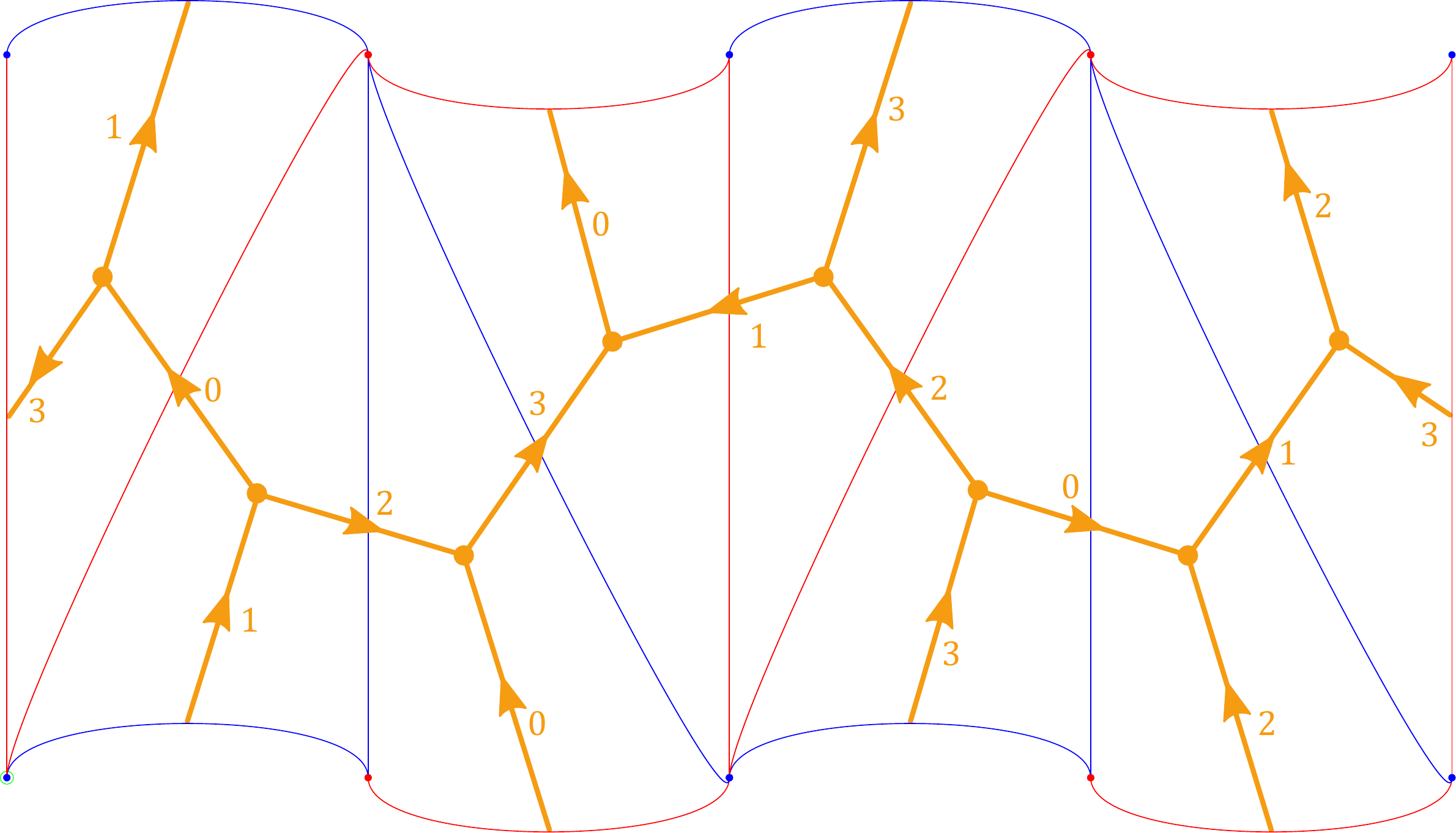}
		\end{center}
		\caption{The boundary track $\beta$ for the veering triangulation \texttt{cPcbbbiht\_12} of the figure eight knot complement and its dual graph~$X^\beta$ (in orange). The orientation on the edges of $X^\beta$ is determined by the coorientation on their dual branches of $\beta$. Edges of $X^\beta$ are labelled with indices of triangles that they pass through.  		 		
			The picture of the boundary track is taken from \cite{VeeringCensus}. The dual boundary graph has been added by the author.}
		\label{fig:boundary_torus_m004} 
	\end{figure}

	The dual boundary graph is a combinatorial tool that we use to encode paths which are transverse to the the boundary track. Moreover, every cycle in the boundary graph can be homotoped inside $M$ to a cycle in the dual graph. 
	\begin{lemma}\label{lem:cycle_in_boundary}
		Let $(\mathcal{T}, \alpha)$ be a transverse taut triangulation of a 3-manifold $M$. Denote by $X$, $X^\beta$ its oriented dual graph and its oriented dual boundary graph, respectively.
		Let~$c^\beta$ be a cycle in $X^\beta$. Suppose it passes consecutively through the edges of $X^\beta$ labelled with $f_{i_1}, \ldots, f_{i_l}$, where $1\leq i_j \leq 2n$. 
		
		We set
		\[s_{i_j} = \begin{cases}
			+1 & \text{if $c$ passes through an edge labelled with $f_{i_j}$ upwards}\\
			-1 & \text{if $c$ passes through an edge labelled with $f_{i_j}$ downwards}.
		\end{cases} \]
		
		Let $c$ be the cycle $(s_{i_1} f_{i_1}, \ldots, s_{i_l} f_{i_l})$ in the dual graph $X$. If we embed $X^\beta$ and $X$ in~$M$ in the natural way, then $c^\beta$ and $c$ are homotopic.
	\end{lemma}
	\begin{proof}
		A homotopy between $c^\beta$ and $c$ can be obtained by pushing each edge of the cycle $c^\beta$ towards the middle of the triangle through which it passes; see Figure \ref{fig:pushing_away}.
	\end{proof}
		\begin{figure}[h]
		\begin{center}
			\includegraphics[width=0.24\textwidth]{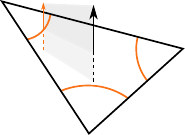}
		\end{center}
		\caption{Homotoping a dual boundary cycle to a dual cycle. The black arrow is an edge of $X$ dual to $f$. The orange arrow is an edge of $X^\beta$ labelled with $f$.}
		\label{fig:pushing_away} 
	\end{figure}
		
	Fix an integer $j$ between 1 and $k$.  The curve $\gamma_j = \partial S^w \cap T_j$  is contained in the boundary track~$\beta_j$. Let $\epsilon$ be a branch of $\beta_j$. 
	Let $s^-$ and $s^+$ be the initial and the terminal switches of $\epsilon$, respectively. We replace each subarc of $\gamma_j$ contained in $\epsilon$ by the following 1-chain $c_\epsilon$ in $X^\beta_j$
	\[
	c_\epsilon = 	- \begin{pmatrix}\text{outgoing branches}\\ \text{of $s^-$ above $\epsilon$}\end{pmatrix} + \begin{pmatrix}\text{incoming branches}\\ \text{ of $s^+$ above $\epsilon$}\end{pmatrix}.
	\]
	This is schematically depicted in Figure \ref{fig:pushing_up}. 
	
	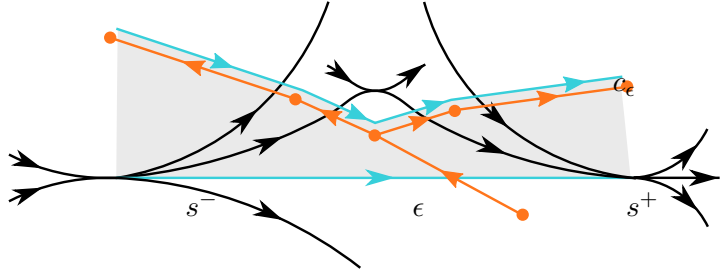
\begin{figure}[h]
		\begin{center}
			\includegraphics[width=0.65\textwidth]{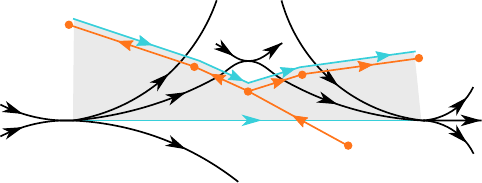}
			\put(-115,20){$\epsilon$}
			\put(-200,20){$s^-$}
			\put(-35,20){$s^+$}
			\put(-40,67){$c_\epsilon$}
		\end{center}
		\caption{A local picture of the boundary track $\beta$ (in black) and the dual boundary graph $X^\beta$ (in orange). We push the branch $\epsilon$ of~$\beta$ (light blue) upwards to the 1-chain $c_\epsilon$ in $X^\beta$ (also light blue).}
		\label{fig:pushing_up} 
	\end{figure}
	
	Let us denote the transverse measure on $\epsilon$ determined by $w = (w_1, \ldots, w_{2n})$ by $w(\epsilon)$. 
	The curve $\gamma_j$ passes through $\epsilon$ $w(\epsilon)$ times. Since  chain groups are abelian, the 1-cycle in $X^\beta$ homotopic to $\gamma_j$ is given by
	\[c_j^\beta = \sum\limits_{\epsilon \subset \beta_j} w(\epsilon) c_\epsilon,\]
	where the sum is over all branches $\epsilon$ of $\beta_j$.
	By Lemma \ref{lem:cycle_in_boundary} we can homotope the 1-cycles $c_j^\beta$ in $X^\beta$ to 1-cycles $c_j$ in $X$.
	
	The procedure explained in this subsection is summed up in algorithm \texttt{Boundary} \texttt{Cycles} below. The algorithm is due to Saul Schleimer and Henry Segerman. We include it here, with permission, for completeness. 
	
	In the algorithm we use the notion of \emph{upward} and \emph{downward} edges. They are defined as follows. A vertex $v$ of an ideal triangle $f \in F$ gives a branch $\epsilon_v$ of $\beta$. We say that an edge $e$ of $f$ is the \emph{downward} edge for $v$ in $f$ if its intersection with $\partial M$ is the initial switch of $\epsilon_v$. An edge $e$ of $f$ is the \emph{upward} edge for~$v$ in $f$ if its intersection with $\partial M$ is the terminal switch of $\epsilon_v$. The names reflect the fact that when we homotope the branch~$\epsilon_v$ to a 1-chain in $X^\beta$ we go downwards above the initial switch of $\epsilon_v$ and upwards above the terminal switch of $\epsilon_v$; see Figure \ref{fig:pushing_up}.

	\begin{figure}[h]
		\begin{center}
			\includegraphics[width=0.21\textwidth]{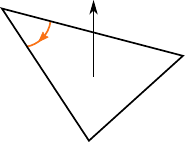}
			\put(-95,60){$v$}
			\put(-30,55){downward edge}
			\put(-130,20){upward edge}
		\end{center}
		\caption{Downward and upward edges for an ideal vertex $v$ of a triangle.}
		\label{fig:cooriented_triangle_edges} 
	\end{figure}
	\begin{algorithm}[H]
		\caption*{\textbf{Algorithm} \texttt{BoundaryCycles} \\ Expressing boundary components of a surface carried by a transverse taut triangulation $(\mathcal{T}, \alpha)$ as simplicial 1-cycles in the dual graph $X$}\label{alg:boundaryCycles}
		\begin{algorithmic}[1]
			\Statex 	\textbf{Input:} \begin{itemize} \item A transverse taut triangulation $\mathcal{T} = ((T,F, E),\alpha)$  of a cusped 3-manifold $M$ with $n$ tetrahedra and $b$ ideal vertices  \item A nonzero tuple $w \in \zz^{2n}$ of integral nonnegative weights on elements of $F$\end{itemize}
			\Statex 	\textbf{Output:} \begin{itemize}
				\item 
				List of $b$ vectors from $\zz^{2n}$, each encoding a simplicial 1-cycle $c_j$ in~$X$ homotopic to $\partial S^w \cap T_j$, for $1\leq j \leq b$\end{itemize}
			\State Boundaries := the list of $j$ zero vectors from $\zz^{2n}$
			\For{$f$ in $F$}
			\For{vertex $v$  of $f$}
			\State $j :=$ the index of $v$ as an ideal vertex of $\mathcal{T}$
			\State $e_1, e_2:$ = the downward and upward edges of $v$ in $f$
			\For {$f'$ above $f$ at the same side of $e_1$}
			\State subtract $w(f)$ from the entry $f'$ of Boundaries[$j$]
			\EndFor
			\For {$f'$ above $f$ at the same side of $e_2$}
			\State add $w(f)$ to the entry $f'$ of Boundaries[$j$]
						\Statex \hfill \color{gray} continued on the next page
									\algstore{part1:cycles}
					\end{algorithmic}
				\end{algorithm}
			\begin{algorithm}[H]
			\caption*{\textbf{Algorithm} \texttt{BoundaryCycles} continued}
			\begin{algorithmic}[1]
			\algrestore{part1:cycles}
			\EndFor
			\EndFor
			\EndFor
			\State \Return Boundaries
		\end{algorithmic}
	\end{algorithm}
	\begin{remark*}
		In this section we  considered only layered triangulations, because the Teichm\"uller polynomial is defined only for fibred faces of the Thurston norm ball. However, algorithm \texttt{BoundaryCycles} can be applied to a measurable triangulation. By \cite[Theorem 5.12]{LMT} a measurable veering triangulation determines a non-fibred face of the Thurston norm ball. 
	\end{remark*}

	\subsection{Computing the Teichm\"uller polynomial}\label{subsec:Teich_comp}
	In this subsection we finally give an algorithm to compute the Teichm\"uller polynomial of any fibred face of the Thurston norm ball. 
	
	By \texttt{Veering} we denote an algorithm which given a pseudo-Anosov homeomorphism $\psi:S \rightarrow S$ outputs 
	\begin{itemize}
		\item the veering triangulation $\mathcal{V}$ of the mapping torus of $\mathring{\psi}:\mathring{S} \rightarrow \mathring{S}$, where $\mathring{S}$ is obtained from $S$ by puncturing it at the singularities of $\psi$,
		\item a nonnegative solution $w=(w_1,\ldots, w_{2n})$ to the system of branch equations of~$\mathcal{V}$ such that  \(S^w = \sum\limits_{i=1}^{2n}w_if_i\) is homologous to the fibre $\mathring{S}$. 
	\end{itemize}Algorithm \texttt{Veering} is explained in \cite[Section 4]{Agol_veer}. It has been implemented by Mark Bell in \texttt{flipper} \cite{flipper}.
	\begin{algorithm}[H]
		\caption*{\textbf{Algorithm} \texttt{Teichm\"ullerPolynomial} \\ Computing the Teichm\"uller polynomial of a fibred face of the Thurston norm ball \label{alg:Teich}}
		\begin{algorithmic}[1]
			\Statex 	\textbf{Input:} 
			A pseudo-Anosov homeomorphism $\psi:S\rightarrow S$ 
			
			\Statex 	\textbf{Output:} The Teichm\"uller polynomial of the face $\face{F}$ in $H_2(N, \partial N;\rr)$, where $N$ is the mapping torus of $\psi$ and $\lbrack S \rbrack \in \rr_+\hspace{-0.1cm}\cdot \face{F}$
			\State $(\mathcal{V},w):$ = \texttt{Veering}($\psi$) 
			\State permute the vertices of $\mathcal{V}$ so that the first $k$ correspond to the torus boundary components of the underlying manifold of $\mathcal{V}$ which are filled in $N$
			\State $n$ : = number of tetrahedra of $\mathcal{V}$
			\State $\Theta: = $ \texttt{TautPolynomial}$(\mathcal{V})$
			\State exp:= \texttt{Exponents}($\Theta$)
			\State coeff:= \texttt{Coefficients}($\Theta$)
			\State $U,r$ := \texttt{FacePairings}($\mathcal{V}, \lbrack \ \rbrack$, return type = ``matrix'')
			\State $A_1$ := the matrix obtained from $U^{-1}$ by deletings its first $n+1-r$ columns
			\State $C:= \texttt{BoundaryCycles}(\mathcal{V}, w)$
			\State $C:= C\lbrack 1:k \rbrack $
			\State $U',s$ := \texttt{FacePairings}($\mathcal{V}, C$, return type = ``matrix'') \label{line:HCpair}
			\State $A_2$ := the matrix obtained from $U'$ by deletings its first $n+1-s$ rows
			\State newExp := $\lbrack \ \rbrack$
			\For{$v$  in exp}
			\State append newExp with $A_2 A_1 (v)$
			\EndFor
			\State $\Theta:=$  $\sum\limits_{i=1}^{\text{length}(\text{coeff})} \text{coeff}\lbrack i \rbrack \cdot u^{\text{newExp}\lbrack i\rbrack}$\Comment{$u=(u_1, \ldots, u_s)$ and $u^v = u_1^{v_1}\cdots u_s^{v_s}$}
			\State	\Return $ \Theta$
		\end{algorithmic}
	\end{algorithm}
	\begin{proposition}\label{prop:alg_teich_correct}
		Let $\psi:S \rightarrow S$ be a pseudo-Anosov homeomorphism. Denote by $N$  its mapping torus. Let $\face{F}$ be the fibred face of the Thurston norm ball in $H_2(N, \partial N;\zz)$ such that  $\lbrack S \rbrack \in \rr_+ \hspace{-0.1cm}\cdot \face{F}$. Then the output of \emph{\texttt{Teichm\"ullerPolynomial}}$(\psi)$ is equal to the Teichm\"uller polynomial $\Theta_{\emph{\texttt{F}}}$ of~$\face{F}$.
	\end{proposition}
	\begin{proof}	
		Let $\mathring{S}$ denote a surface obtained from $S$ by puncturing it at the singularities of the invariant foliations of~$\psi$. The pair $(\mathcal{V}, w)$ in line 1 consists of the veering triangulation of $M =  N \bez \mathrm{sing}(\face{F})$ associated to $\face{F}$ and a nonnegative solution to its system of branch equations which realises $\mathring{S}$. We permute the vertices of $\mathcal{V}$ so that the first $k$ correspond to the torus boundary components $T_1, \ldots, T_k$ of~$M$ (in the compact model) which are filled in~$N$. Then the list $C$ constructed in line 10 consists of dual cycles homologous to the  boundary components of $\mathring{S}=S^w$ which are contained in $T_1, \ldots, T_k$. Therefore $H_N=H_M^C$. 
		
		By Proposition~\ref{prop:veering_output} the algorithm \texttt{TautPolynomial}($\mathcal{V}$) outputs  the taut polynomial~$\Theta_{\mathcal{V}}$. As explained in Subsection \ref{subsec:epimorphism}, the matrix $A_2A_1$ represents the epimorphism $i_\ast: H_M \rightarrow H_M^C$, where the basis for $H_M$ is the same as the one we use for the computation of the taut polynomial. Each monomial $a_h\cdot h$ in~$\Theta_{\mathcal{V}}$ can be encoded by a pair $(a_h, v)$ where  $a_h \in \zz$, $v \in \zz^r$. The pair $(a_h, A_2A_1(v))$ then encodes the corresponding monomial $a_h\cdot i_\ast(h)$ appearing in $i_\ast(\Theta_\mathcal{V})$. Therefore by Lemma \ref{lem:computing_Teich} the polynomial $\Theta$ at line 17 is equal to $\Theta_{\mathtt{F}}$.	
	\end{proof}

	\bibliographystyle{abbrv}
	\bibliography{mybib}

\begin{thebibliography}{10}

\bibitem{Agol_veer}
I.~Agol.
\newblock Ideal triangulations of pseudo-{A}nosov mapping tori.
\newblock In W.~Li, L.~Bartolini, J.~Johnson, F.~Luo, R.~Myers, and J.~H.
  Rubinstein, editors, {\em Topology and Geometry in Dimension Three:
  Triangulations, Invariants, and Geometric Structures}, volume 560 of {\em
  Contemporary Mathematics}, pages 1--19. American Mathematical Society, 2011.

\bibitem{BWKJ}
H.~Baik, C.~Wu, K.~Kim, and T.~Jo.
\newblock An algorithm to compute the {T}eichm{\"u}ller polynomial from
  matrices.
\newblock {\em Geometriae Dedicata}, 204:175--189, 2020.

\bibitem{flipper}
M.~Bell.
\newblock \texttt{flipper} (computer software).
\newblock \url{pypi.python.org/pypi/flipper}, 2013--2020.

\bibitem{BilletLechti}
R.~Billet and L.~Lechti.
\newblock Teichm{\"u}ller polynomials of fibered alternating links.
\newblock {\em Osaka J. Math.}, 56(4):787--806, 2019.

\bibitem{Burton_isoSig}
B.~A. Burton.
\newblock The {P}achner graph and the simplification of 3-sphere
  triangulations.
\newblock In {\em Proceedings of the {T}wenty-{S}eventh {A}nnual {S}ymposium on
  {C}omputational {G}eometry}, pages 153--162. Association for {Co}mputing
  {M}achinery, 2011.

\bibitem{crow_fox}
R.~Crowell and R.~Fox.
\newblock {\em Introduction to {K}not {T}heory}, volume~57 of {\em Graduate
  Texts in Mathematics}.
\newblock Springer-Verlag New York, 1963.

\bibitem{Fried_suspension}
D.~Fried.
\newblock Fibrations over ${S}^1$ with {P}seudo-{A}nosov {M}onodromy.
\newblock In A.~Fathi, F.~Laudenbach, and V.~Po\'enaru, editors, {\em
  Thurston's work on surfaces}, chapter~14, pages 215--230. Princeton
  University Press, 2012.

\bibitem{explicit_angle}
D.~Futer and F.~Gu\'eritaud.
\newblock Explicit angle structures for veering triangulations.
\newblock {\em Algebr. Geom. Topol.}, 13(1):205--235, 2013.

\bibitem{VeeringCensus}
A.~Giannopolous, S.~Schleimer, and H.~Segerman.
\newblock A census of veering structures.
\newblock \url{https://math.okstate.edu/people/segerman/veering.html}.

\bibitem{veer_strict-angles}
C.~D. Hodgson, J.~H. Rubinstein, H.~Segerman, and S.~Tillmann.
\newblock Veering triangulations admit strict angle structures.
\newblock {\em Geometry \& Topology}, 15(4):2073--2089, 2011.

\bibitem{Lack_taut}
M.~Lackenby.
\newblock Taut ideal triangulations of 3-manifolds.
\newblock {\em Geometry \& Topology}, 4(1):369--395, 2000.

\bibitem{LMT_flow}
M.~Landry, Y.~N. Minsky, and S.~J. Taylor.
\newblock Flows, growth rates, and the veering polynomial.
\newblock In preparation.

\bibitem{LMT}
M.~Landry, Y.~N. Minsky, and S.~J. Taylor.
\newblock A polynomial invariant for veering triangulations.
\newblock arXiv:2008.04836 [math.GT].

\bibitem{LanVal}
E.~Lanneau and F.~Valdez.
\newblock Computing the {T}eichm{\"u}ller polynomial.
\newblock {\em Journal of the European Mathematical Society},
  19(12):3867--3910, 2017.

\bibitem{McMullen_Teich}
C.~T. McMullen.
\newblock Polynomial invariants for fibered 3-manifolds and {T}eichm{\"u}ller
  geodesics for foliations.
\newblock {\em Ann. Scient. Éc. Norm. Sup.}, 33(4):519--560, 2000.

\bibitem{McMullen_Perron}
C.~T. McMullen.
\newblock Entropy and the clique polynomial.
\newblock {\em Journal of Topology}, 8(1):184--212, 2015.

\bibitem{MinskyTaylor}
Y.~N. Minsky and S.~J. Taylor.
\newblock Fibered faces, veering triangulations, and the arc complex.
\newblock {\em Geom. Funct. Anal.}, 27(6):1450--1496, 2017.

\bibitem{Northcott}
D.~Northcott.
\newblock {\em Finite Free Resolutions}.
\newblock Cambridge Tracts in Mathematics. Cambridge University Press, 1976.

\bibitem{Oertel_lam}
U.~Oertel.
\newblock Measured {L}aminations in 3-{M}anifolds.
\newblock {\em Transactions of the American Mathematical Society},
  305(2):531--573, 1988.

\bibitem{Penner_tt}
R.~Penner and J.~Harer.
\newblock {\em Combinatorics of train tracks}.
\newblock Number 125 in Annals of {M}athematics Studies. Princeton {U}niversity
  {P}ress, 1992.

\bibitem{SchleimSegLinks}
S.~Schleimer and H.~Segerman.
\newblock From veering triangulations to link spaces and back again.
\newblock arXiv:1911.00006 [math.GT].

\bibitem{Thur_norm}
W.~P. Thurston.
\newblock A norm for the homology of 3-manifolds.
\newblock {\em Memoirs of the American Mathematical Society}, 59(339):100--130,
  1986.

\bibitem{Thurston_map_tor}
W.~P. Thurston.
\newblock Hyperbolic structures on 3-manifolds, {II}: {S}urface groups and
  3-manifolds which fiber over the circle.
\newblock arXiv:math/9801045 [math.GT], 1998.

\bibitem{Traldi}
L.~Traldi.
\newblock The determinantal ideals of link modules.
\newblock {\em Pacific J. Math}, 101(1):215--222, 1982.

\end{thebibliography}
\end{document}